\documentclass[12pt,reqno]{amsart}
\usepackage{amsmath, amssymb, amsthm}
\usepackage{url}
\usepackage[breaklinks]{hyperref}
\renewcommand{\Im}{\operatorname{Im}}

\renewcommand{\Re}{\operatorname{Re}}
\renewcommand{\Im}{\operatorname{Im}}

\renewcommand{\(}{\left\(}
\renewcommand{\)}{\right\)}
\renewcommand{\[}{\left\[}
\renewcommand{\]}{\right\]}

\numberwithin{equation}{section}
 \theoremstyle{plain}
\newtheorem{theorem}{Theorem}[section]
\newtheorem{lemma}[theorem]{Lemma}
\newtheorem{remark}[]{Remark}

\usepackage{tikz}
\usetikzlibrary{decorations.markings}

\newtheorem{conjecture}[theorem]{Conjecture}

\newtheorem{corollary}[theorem]{Corollary}

\newtheorem{example}[]{Example}

   \makeatletter
\def\proof{\@ifnextchar[{\@oproof}{\@nproof}}
\def\@oproof[#1][#2]{\trivlist\item[\hskip\labelsep\textit{#2 Proof of\
#1.}~]\ignorespaces}
\def\@nproof{\trivlist\item[\hskip\labelsep\textit{Proof.}~]\ignorespaces}
\makeatother

\begin{document}
\title[Equivalent criteria for the Riemann hypothesis for a general class of $L$-functions]{Equivalent criteria for the Riemann hypothesis for a general class of $L$-functions}

\author{Meghali Garg}
\address{Meghali Garg \\ Department of Mathematics \\
Indian Institute of Technology Indore \\
Simrol,  Indore,  Madhya Pradesh 453552, India.} 
\email{meghaligarg.2216@gmail.com,   phd2001241005@iiti.ac.in}

 \author{Bibekananda Maji}
\address{Bibekananda Maji\\ Department of Mathematics \\
Indian Institute of Technology Indore \\
Simrol,  Indore,  Madhya Pradesh 453552, India.} 
\email{bibek10iitb@gmail.com,  bmaji@iiti.ac.in}


\maketitle
%\begin{abstract}
%In 1916,  Hardy and Littlewood established an interesting equivalent criteria for the Riemann hypothesis while correcting an identity of Ramanujan.  A few years back,  Dixit obtained a character analogue of the aforementioned identity of Hardy and Littlewood.  In the current paper,  we  obtain a one-variable generalization of the identity of Dixit.  This generalization also motivates us to derive an equivalent criteria for the generalized Riemann hypothesis.  At the end of the paper,  we also indicate that a similar equivalent criteria might be true for the grand Riemann hypothesis.  
%\end{abstract}

\begin{abstract}
In 1916,  Riesz gave an equivalent criterion for the Riemann hypothesis (RH).  Inspired from Riesz's criterion,  Hardy and Littlewood showed that RH is equivalent to the following bound: 
\begin{align*}
P_1(x):= \sum_{n=1}^\infty  \frac{\mu(n)}{n} \exp\left({-\frac{x}{n^2}}\right) = O_{\epsilon}\left( x^{-\frac{1}{4}+ \epsilon } \right), \quad \mathrm{as}\,\, x \rightarrow \infty.  
\end{align*}
Recently,  the authors extended the above bound for the generalized Riemann hypothesis for Dirichlet $L$-functions and gave a conjecture for a class of ``nice'' $L$-functions.  In this paper,  we settle this conjecture.  In particular,  we give equivalent criteria for the Riemann hypothesis for $L$-functions associated to cusp forms.   We also obtain an entirely novel form of equivalent criteria for the Riemann hypothesis of $\zeta(s)$.  
Furthermore,  we generalize an identity of Ramanujan,  Hardy and Littlewood for Chandrasekharan-Narasimhan class of $L$-functions.

\end{abstract}

\section{introduction}

The Riemann hypothesis, formulated by Bernhard Riemann in his seminal paper \cite{Rie59} in 1859, stands as one of the most renowned yet unsolved problems in the realm of mathematics. Until now, no counterexample has been discovered, and the hypothesis receives substantial numerical support. Nevertheless, despite the persistent endeavours of mathematicians over time, the Riemann hypothesis continues to elude proof. Over the course of time, mathematicians have put forth numerous equivalent criterion for the Riemann hypothesis in their attempts to establish its proof. A noteworthy equivalent criterion, attributed to Riesz \cite{Riesz} in 1916 has shown that the Riemann hypothesis is equivalent to the following bound for an infinite series associated with $\mu(n)$,
\begin{equation}\label{Riesz}
P_2(x):=  \sum_{n=1}^\infty \frac{\mu(n)}{n^2} \exp\left( - \frac{x}{n^2} \right) = O_{\epsilon} \left( x^{-\frac{3}{4} + \epsilon} \right), \quad {\rm as}\,\,  x \rightarrow \infty,
\end{equation}
for any positive $\epsilon$. Driven by this motivation, Hardy and Littlewood \cite[p.~161]{HL-1916} established another equivalent criterion for the Riemann hypothesis while rectifying an identity of Ramanujan.  Mainly,  they showed that
\begin{align}\label{Riesz type_Hardy_Littlewood}
P_1(x):= \sum_{n=1}^\infty  \frac{\mu(n)}{n} \exp\left({-\frac{x}{n^2}}\right) = O_{\epsilon}\left( x^{-\frac{1}{4}+ \epsilon } \right), \quad \mathrm{as}\,\, x \rightarrow \infty.  
\end{align}
Their intuition stemmed from the following identity found in Ramanujan's notebooks \cite[p.~312]{Rama_2nd_Notebook},  \cite[Equation (37.3), p.~470]{BCB-V}, where he mentioned that for any  $x>0$,
\begin{align}\label{Ramanujan_mu}
\sum_{n=1}^{\infty} \frac{\mu(n)}{n} \exp\left({-\frac{x}{n^2}}\right) = \sqrt{\frac{\pi}{x}} \sum_{n=1}^{\infty} \frac{\mu(n)}{n} \exp\left(- \frac{\pi^2}{n^2 x} \right).
\end{align}
Berndt \cite[p.~468-469]{BCB-V}  presented a compelling numerical explanation that sheds light on the discrepancy of the identity \eqref{Ramanujan_mu}.
 Hardy and Littlewood  \cite[p.~156, Section 2.5]{HL-1916}  established a corrected version of \eqref{Ramanujan_mu}.  They proved that,  for $x>0$,  
 \begin{align}\label{Hardy-Littlewood}
 \sum_{n=1}^{\infty} \frac{\mu(n)}{n} \exp\left({-\frac{x}{n^2}}\right) = \sqrt{\frac{\pi}{x}} \sum_{n=1}^{\infty} \frac{\mu(n)}{n} \exp\left(- \frac{\pi^2}{n^2 x} \right)- \frac{1}{2 \sqrt{\pi}} \sum_{\rho} \left(  \frac{\pi}{\sqrt{x}} \right)^{\rho}  \frac{ \Gamma\left(\frac{1-\rho}{2} \right) }{\zeta'(\rho)}.
 \end{align}
 A symmetric form of the above identity is the following.  For two positive real numbers $\alpha$,  $\beta$ with $\alpha \beta =\pi$,  we have
 \begin{align}\label{Hardy-Littlewood symmetric}
\sqrt{\alpha} \sum_{n=1}^{\infty} \frac{\mu(n)}{n} \exp\left({-\left(\frac{\alpha}{n}\right)^2}\right) - \sqrt{\beta} \sum_{n=1}^{\infty} \frac{\mu(n)}{n} \exp\left({-\left(\frac{\beta}{n}\right)^2}\right) = -\frac{1}{2\sqrt{\beta}} \sum_{\rho} \frac{ \Gamma\left(\frac{1-\rho}{2} \right) \beta^\rho}{\zeta'(\rho)},  
\end{align}
where the right-side sum over $\rho$ runs through the non-trivial zeros of the Riemann zeta function $\zeta(s)$.  The identity \eqref{Hardy-Littlewood} holds under the assumption that, all the non-trivial zeros of $\zeta(s)$ are simple. The convergence of the sum over $\rho$ on the right-hand side of \eqref{Hardy-Littlewood} is not immediately evident. While there is a belief that the series converges rapidly, to date, its convergence is only established under the assumption of bracketing the terms. That is, if for some positive constant $A_0$, the terms corresponding to the non-trivial zeros $\rho_1$ and $\rho_2$ satisfy 
\begin{align}\label{bracketing}
|\Im(\rho_1) - \Im(\rho_2)| < \exp \left( -\frac{A_0 \Im(\rho_1)}{\log(\Im(\rho_1) )} \right) + \exp \left( -\frac{A_0 \Im(\rho_2)}{\log(\Im(\rho_2) )} \right),
\end{align}
then they will be considered inside the same bracket. A more detailed explanation of the convergence intricacies concerning series containing $\zeta'(\rho)$ in the denominator has been provided in \cite[p.~783]{JMS21}.

The identity \eqref{Hardy-Littlewood} has captured the attention of numerous mathematicians over the years. For further insights into this identity, interested readers are encouraged to refer to Berndt \cite[p.~470]{BCB-V}, Paris and Kaminski \cite[p.~143]{PK}, and Titchmarch \cite[p.~219]{Tit}. Remarkably, Bhaskaran \cite{bhas} discovered a connection between Fourier reciprocity and Wiener's Tauberian theory. Ramanujan himself indicated a fascinating generalization of \eqref{Hardy-Littlewood} related to a pair of reciprocal functions, which was later explored by Hardy and Littlewood \cite[p.~160, Section 2.5]{HL-1916}.

In 2012, Dixit \cite{dixit12} established an elegant character analogue of \eqref{Hardy-Littlewood}. Subsequently, Dixit, Roy, and Zaharescu \cite{DRZ-character} investigated a one-variable generalization of \eqref{Hardy-Littlewood}, deriving equivalent criteria for both the Riemann hypothesis and the generalized Riemann hypothesis. They \cite{DRZ} also found an identity analogous to \eqref{Hardy-Littlewood} for $L$-functions associated with Hecke eigenforms. In $2022$, Dixit, Gupta, and Vatwani \cite{DGV21} obtained a generalization for the Dedekind zeta function and established an equivalent criterion for the extended Riemann hypothesis.  In 2023,  Banerjee and Kumar \cite{BK21} derived an analogous identity for $L$-functions associated to Maass cusp forms and discovered Riesz-Hardy-Littlewood type equivalent criterion for the corresponding $L$-function. Very recently, Gupta and Vatwani \cite{GV22} further examined this identity for $L$-functions in the Selberg class and obtained Riesz type equivalent criteria.  Interested people can also explore the work of Báez-Duarte \cite{Baez1}, where a sequential Riesz-type criterion for the Riemann hypothesis was presented. This phenomenon was further studied by Cislo and Wolf \cite{CW08}.

In a recent development, Agarwal and the authors \cite{AGM}  gave an interesting one-variable generalization of the identity of Hardy-Littlewood \eqref{Hardy-Littlewood}. The following identity has been established.
\begin{theorem}
Let $k \geq 1$ be a real number.  Under the hypothesis that the non-trivial zeros of $\zeta(s)$ are simple,  for $x>0$,  we have
\begin{align}\label{AGM}
    \sum_{n=1}^{\infty} \frac{\mu(n)}{n^k} \exp\left({-\frac{x}{n^2}}\right) =\frac{\Gamma(\frac{k}{2})}{x^\frac{k}{2}}\sum_{n=1}^{\infty}\frac{\mu(n)}{n}{}_1F_1 \left(\frac{k}{2}; \frac{1}{2}; - \frac{\pi^2}{n^2x} \right) + \frac{1}{2}\sum_{\rho}\frac{\Gamma(\frac{k-\rho}{2})}{\zeta'(\rho)}x^{-\frac{(k-\rho)}{2}},
\end{align}
where the sum over $\rho$ runs through all the non-trivial zeros of $\zeta(s)$ and satisfies the bracketing condition \eqref{bracketing} and the entire function ${}_1F_1(a;b;z)$ is defined by 
\begin{align*}
{}_1F_1(a;b;z):= \sum_{n=0}^\infty \frac{(a)_n}{(b)_n} \frac{z^n}{n!},
\quad
{where}~~ (a)_n:=\frac{\Gamma(a+n)}{\Gamma(a)}.
\end{align*}
\end{theorem}
%$ {}_1F_{1} \left( \begin{matrix} a \\ b \end{matrix}  \Big| z \right)$ 

%\begin{remark}
%The function ${}_1F_1$ on the right hand side of \eqref{AGM} is a particular case of generalized hypergeometric series \cite[p. 404, Equation 16.2.1]{NIST} and is defined by the following series expansion
%\begin{align*}
%{}_pF_q \left( c_1, \cdots, c_p;\,  d_1, \cdots, d_q;\, z \right):= \sum_{n=0}^{\infty} \frac{(c_1)_n \cdots (c_p)_n}{(d_1)_n\cdots (d_q)_n} \frac{z^n}{n!},
%\end{align*}
%where $(c)_n:=\frac{\Gamma(c+n)}{\Gamma(c)}$ and $c_1, \cdots, c_p$ and $d_1, \cdots, d_q$ be  $p+q$ complex numbers. If  $p \leq q$, this series converges for all complex values of $z$. However, if $p=q+1$,  then it converges for $|z|<1$, though for $|z|<1$, it can be analytically continued to $ \mathbb{C}$ if we consider a branch cut from $1$ to $+\infty$. Again, if $p=q=1$,  that is,  ${}_1F_1 \left( c_1 ;\,  d_1;\, z \right)$ is an entire function.
%\end{remark}
%Setting $k=1$ and $x=\frac{\pi^2}{x}$ in \eqref{AGM}, it becomes straight forward to derive the identity \eqref{Hardy-Littlewood}. 

Building upon Dixit's  \cite{dixit12} work, the authors \cite{GM-character} provided a character analogue of \eqref{AGM} and established the identity as follows:
\begin{theorem}
Let $k \geq 1$ and $\chi$ be a  primitive character modulo $q$.  Assume that all the non-trivial zeros of $L(s,\chi)$ are simple.  For $x >0$,  we have
{\allowdisplaybreaks
\begin{align}
      \sum_{n=1}^{\infty} \frac{\chi(n) \mu(n)}{n^k} \exp \left({-\frac{\pi x^2}{q n^2}}\right) & =  \frac{i^a \sqrt{q}}{G(\chi)} \left(\frac{q}{\pi x^2}\right)^{\frac{k+a}{2}} \left( \frac{\pi}{q} \right)^{a+\frac{1}{2}} \frac{\Gamma(\frac{k+a}{2})}{\Gamma(a+\frac{1}{2})} \nonumber \\
     & \times  \sum_{n=1}^{\infty} \frac{\overline{\chi(n)} \mu(n)}{n^{1+a}} {}_1F_{1} \left( \frac{k+a}{2}; a+\frac{1}{2}; - \frac{\pi}{qn^2 x^2} \right) \nonumber \\
     &+ \frac{1}{2}\sum_{ \rho } \frac{\Gamma(\frac{k-\rho}{2})}{ L'(\rho , \chi)} \left(\frac{q}{\pi x^2}\right)^{\frac{k-\rho}{2}},  \label{character_analogue}
\end{align}}
where the sum over $\rho$ runs through the non-trivial zeros of $L(s,\chi)$ and satisfy the bracketing conditions \eqref{bracketing}.
\end{theorem}
Inspired by the works of Riesz, Hardy and Littlewood, the authors \cite{GM-character} have established the following equivalent criteria for the generalized Riemann hypothesis for $L(s, \chi)$. For $k\geq 1$ and $\ell >0$,  we have
\begin{equation}\label{Two_variable_GM_bound_Final}
  \sum_{n=1}^{\infty} \frac{\chi(n) \mu(n)}{n^{k}} \exp \left(- \frac{ x}{n^{\ell}}\right) = O_{\epsilon,k,\ell} \left(x^{-\frac{k}{\ell}+\frac{1}{2 \ell} + \epsilon }\right), \quad \mathrm{as}\,\, x \rightarrow \infty, 
\end{equation}
for any $\epsilon > 0$. This bound extends the previously established bounds by Riesz \eqref{Riesz}, Hardy-Littlewood \eqref{Riesz type_Hardy_Littlewood}, as well as the earlier bound given by Agarwal and the authors \cite[Theorem 1.2]{AGM}. This bound, in turn, has motivated us to propose a broader conjecture \cite[Conjecture 5.1]{GM-character},  which we restate below for the clarity of the readers.  Let us consider
\begin{align}\label{def nyc L function}
L(f,s) = \sum_{n=1}^{\infty}\frac{A_f(n)}{n^s}, ~{\rm{for}}~ \Re(s) > 1,
\end{align}
be a ``nice" $L$-function.  Mainly,  we are assuming that $L(f,s)$ obeys the grand Riemann hypothesis \cite[p.~113]{IK},  which says that all the non-trivial zeros of $L(f,s)$ in the critical strip $0 <\Re(s)<1$ lie on the critical line $\Re(s)= \frac{1}{2}$.  Let us suppose that
\begin{align}\label{inverse of L(f,s)}
1/L(f,s) = \sum_{n=1}^\infty \frac{\mu_f(n)}{n^s},~ \text{for}~ \Re(s) > 1. 
\end{align}
Moreover,  we assume that $\sum_{n=1}^\infty \mu_f(n) n^{-1}$ is convergent.  Littlewood's prediction says that the grand Riemann hypothesis \cite[Proposition 5.14]{IK} is equivalent to the following bound: 
\begin{align}\label{merten bound general}
\sum_{  n \leq x} \mu_f(n) \ll_{\epsilon, f} x^{\frac{1}{2} + \epsilon},  
\end{align}
for any $\epsilon>0$.   
The above bound played a significant role to prove \eqref{Two_variable_GM_bound_Final}.  
%Notably, the exact same bound, when it comes to the generalized Riemann hypothesis, was pivotal in deriving the aforementioned bound \eqref{Two_variable_GM_bound_Final}.  
Inspired from this observation,  the authors gave the following conjecture. 
\begin{conjecture}\label{conjecture}
Let $k \geq 1$ and $\ell>0$ be two real numbers.  Let $L(f,s)$ be a ``nice'' $L$-function satisfying $\eqref{merten bound general}$. Then the grand Riemann hypothesis for $L(f,s)$ is equivalent to the bound
\begin{align}\label{conjecture main bound}
P_{k,f,\ell} :=   \sum_{n=1}^{\infty} \frac{\mu_f(n)}{n^{k}} \exp \left(- \frac{ x}{n^{\ell}}\right) = O_{\epsilon, k, \ell} \left(x^{-\frac{k}{\ell}+\frac{1}{2 \ell} + \epsilon }\right), \quad \mathrm{as}\,\, x \rightarrow \infty. 
\end{align}
\end{conjecture}
In this manuscript,  one of our main goals is to prove the above conjecture.  Moreover,  we also establish  a generalization of the Ramanujan-Hardy-Littlewood identity \eqref{Hardy-Littlewood} within the framework of Chandrasekharan-Narasimhan class of $L$-functions \cite{chand},  defined in the next section.   Furthermore,  we obtain an entirely novel form of equivalent criteria for the Riemann hypothesis of $\zeta(s)$,  see Theorem \ref{Equivalent criteria for RH by sigma}. 
%Additionally, we establish  Hardy-Littlewood-Riesz-type equivalent criteria for $L$-functions in this class that satisfy the grand Riemann hypothesis,  which gives the validity of Conjecture \ref{conjecture} for $L(f,s)$ defined as in \eqref{def nyc L function}.
 %where $f$ is a holomorphic Hecke eigenform of weight $\omega$ over the full modular group $SL_2(\mathbb{Z})$.

\section{Chandrasekharan-Narasimhan class of $L$-functions}

We now introduce the class of $L$-functions established by Chandrasekharan and Narasimhan \cite{chand}.
Consider two arithmetical functions, denoted as $a_1(n)$ and $b_1(n)$, which are not identically zero. Additionally, let ${\lambda_n}$ and ${\mu_n}$ be two monotonically increasing sequences of positive real numbers tending to infinity. Then the associated Dirichlet series are defined as follows:
\begin{align}
\phi(s) = \sum_{n = 1}^{\infty}\frac{a_1(n)}{\lambda_{n}^s} \quad {\rm{for}}\quad \Re(s) = \sigma > \sigma_a,  \label{CL Dirichlet series1} \\
\psi(s) = \sum_{n = 1}^{\infty}\frac{b_1(n)}{\mu_{n}^s} \quad {\rm{for}}\quad \Re(s) = \sigma > \sigma_b. \label{CL Dirichlet series2}
\end{align}
Here, $\sigma_a$ and $\sigma_b$ indicate the finite abscissae of absolute convergence for $\phi(s)$ and $\psi(s)$, respectively. It is assumed that both $\phi(s)$ and $\psi(s)$ can be analytically continued to the entire $\mathbb{C}$ except for a finite number of poles.  For a given $\delta > 0$,  we assume that $\phi(s)$ and $\psi(s)$ are connected by the subsequent functional equation:
\begin{align}\label{functional equation for phi and psi}
(2\pi)^{-s} \Gamma(s) \phi(s) = (2\pi)^{-(\delta -s)} \Gamma(\delta -s) \psi(\delta - s).
\end{align}
In this paper, we adopt their framework in a slightly general form stated below.  Essentially, we consider a definition that involves a certain twist.  We define
\begin{align*}
a(n)=\begin{cases}
a_1(k),   &\textit{if} \quad  n = \lambda_k,\\
0 ,  & otherwise,
\end{cases}
\quad \text{and} \quad
b(n)=\begin{cases}
b_1(k),   &\textit{if} \quad n = \mu_k,  \\
0,   & otherwise.
\end{cases}
\end{align*}
Thus,  the  Dirichlet series in \eqref{CL Dirichlet series1} and \eqref{CL Dirichlet series2} can be rewritten as follows: 
%we proceed to define the associated Dirichlet series as follows:
\begin{align*}
\phi(s) = \sum_{n = 1}^{\infty}\frac{a(n)}{n^s}, \quad \text{for}\quad \Re(s) = \sigma > \sigma_a, \\
\psi(s) = \sum_{n = 1}^{\infty}\frac{b(n)}{n^s}, \quad \text{for}\quad \Re(s) = \sigma > \sigma_b.
\end{align*}
Now,  for a specified $\delta > 0$, we assume that $\phi(s)$ and $\psi(s)$ satisfy the following general functional equation:
\begin{align}\label{modified functional equation for phi and psi}
c^s \Gamma(As + B) \phi(s) = \nu c^{\delta -s} \Gamma(A(\delta -s) + B) \psi(\delta - s),
\end{align}
where $|\nu| = 1$, $A>0$, $ B \geq 0$ and $c \in \mathbb{R}^{+}$.
One can easily see that \eqref{functional equation for phi and psi} is a particular case of \eqref{modified functional equation for phi and psi} by letting $c = \frac{1}{2\pi}, A=1, B=0$ and $\nu = 1.$

%\begin{remark}
We can observe that $\zeta(s)$ and $L(s,\chi)$ trivially fall under the category of the above class of $L$-functions.
%\end{remark}
A few more examples that belong to the Chandrasekharan-Narasimhan class of $L$-functions are given below. 
\begin{example}\label{def zeta(s) zeta(s-r)}
Let $r \in \mathbb{Z}$ and $\sigma_r(n)$ denote the sum of $r^{th}$ powers of positive divisors of $n$. Then the $L$-function associated to $\sigma_r(n)$ is given by
\begin{align*}
\ \sum_{n=1}^{\infty}\frac{\sigma_r(n)}{n^s} = \zeta(s) \zeta(s-r),   \quad {\rm{for}}\quad \Re(s) = \sigma > \max\{1, r+1 \},
\end{align*} 
and if $r$ is an odd positive integer, it satisfies the following functional equation:
\begin{align}\label{fn equ zeta(s) zeta(s-r)}
(2\pi)^{-s} \Gamma(s) \zeta(s) \zeta(s-r)=  (-1)^{\frac{r+1}{2}}(2\pi)^{-(r+1-s)} \Gamma(r+1-s) \zeta(r+1-s) \zeta(1-s).
\end{align}
This corresponds to \eqref{modified functional equation for phi and psi} for $a(n) = b(n) = \sigma_r(n),  c = \frac{1}{2\pi}, A=1, B=0, \delta = r+1$ and $\nu = (-1)^{\frac{r+1}{2}}.$ One can refer \cite[Example 2]{chand} for further insight into the functional equation \eqref{fn equ zeta(s) zeta(s-r)}.
\end{example}
An intriguing observation about Example \ref{def zeta(s) zeta(s-r)} is that, despite the function $\zeta(s) \zeta(s-r)$ having non-trivial zeros on the lines $\Re(s) = \frac{1}{2}$ and $\Re(s) = r+\frac{1}{2}$, and thus not adhering its own Riemann hypothesis but it offers surprising insights into the Riemann hypothesis for $\zeta(s)$ by presenting an entirely novel form of equivalent criteria, as demonstrated in Theorem \ref{Equivalent criteria for RH by sigma}. The next example corresponds to the cusp form of weight $12$.
\begin{example}\label{def tau }
 The Ramanujan tau function $\tau(n)$ is defined by
\begin{align*}
\Delta(q)=\sum_{n=1}^{\infty} \tau(n) q^n = q\prod_{n=1}^{\infty}(1-q^n)^{24}, \quad |q|<1. 
\end{align*}
Note that $\Delta(q)$ is a cusp form of level $1$ and weight $12$. The $L$-function associated to normalized coefficients $\tau_0(n) = \tau(n) n^{-\frac{11}{2}}$ is given by
\begin{align*}
L(\Delta,s) =  \sum_{n=1}^{\infty}\frac{\tau_0(n)}{n^s}  \quad {\rm{for}}\quad \Re(s) = \sigma > 1,
\end{align*}
and can be analytically continued to the whole complex plane as an entire function. Moreover, $L(\Delta ,s)$ satisfies the following nice-looking functional equation:
\begin{align*}
(2\pi)^{-s} \Gamma\left(s + \frac{11}{2} \right)L(\Delta,s) = (2\pi)^{-(1-s)} \Gamma\left(\frac{13}{2}-s \right)L(\Delta,1-s),
\end{align*}
which aligns with \eqref{modified functional equation for phi and psi} for  $a(n) = b(n) = \tau_0(n),  c = \frac{1}{2\pi}, A=1, B=\frac{11}{2}$, $ \delta = 1$ and $\nu = 1$. It has been conjectured that all the non-trivial zeros of $L(\Delta,s)$ lie on the critical line $\Re(s)=\frac{1}{2}$.
\end{example}
%In the context of the $L$-function linked to Ramanujan's tau function, the Riemann hypothesis states that all the non-trivial zeros are situated on the critical line $\Re(s) = 6$. Consequently, we utilize normalized coefficients to guarantee that this alignment with the critical line $\Re(s) = \frac{1}{2}$ is maintained.
More general setting for $L$-functions associated to cusp forms are stated below.  
\begin{example}\label{Def cusp}
 Let $f(z)$ be a holomorphic Hecke eigenform of weight $\omega$ for the full modular group $SL(2,\mathbb{Z})$. The Fourier expansion of $f(z)$ at $i \infty$ is given by
\begin{align*}
f(z) = \sum_{n=1}^{\infty} \lambda_{f}(n)n^{\frac{\omega -1}{2}} e^{2\pi i n z},
\end{align*}
and $\lambda_{f}(n)$ is the normalized $n$th Fourier coefficient. Then the $L$-function associated to $f(z)$ is defined by 
\begin{align}\label{cusp form}
L(f,s) = \sum_{n=1}^{\infty} \frac{\lambda_{f}(n)}{n^s} = \prod_{p:~\rm{prime}} \left(1-\frac{\lambda_{f}(p)}{p^s} + \frac{1}{p^{2s}} \right)^{-1} \quad {\rm{for}}\quad \Re(s) > 1.
\end{align}
Hecke proved that $L(f,s)$ satisfies the functional equation 
\begin{align}\label{fn eq cusp form}
\Lambda (f,s) =  \Lambda (f,1-s),
\end{align}
where $\Lambda (f,  s):= (2\pi)^{-s} \Gamma\big(s+\frac{\omega -1}{2}\big) L(f,s)$. The $L$-function $L(f,s)$ admits analytic continuation to the whole complex plane.  Here, \eqref{fn eq cusp form} synchronizes with \eqref{modified functional equation for phi and psi} for  $a(n) = b(n) = \lambda_f(n),  c = \frac{1}{2\pi}, A=1, B=\frac{\omega -1}{2}$, $ \delta = 1$ and $\nu =1$. The Riemann hypothesis for $L(f,s)$ states that all the non-trivial zeros of $L(f,s)$ are situated on the critical line $\Re(s)=\frac{1}{2}$.
\end{example}
\begin{example}\label{r_b(n)}
Let $r_{j}(n)$ counts the number of ways a positive integer $n$ can be written as sum of $j$ squares. Then the generating function 
\begin{align*}
\zeta_{id}(j,s) := \sum_{n=1}^{\infty} \frac{r_{j}(n)}{n^s}  \quad {\rm{for}}\quad \Re(s) > \frac{j}{2}, 
\end{align*}
is the Epstein zeta function and satisfies the functional equation
\begin{align*}
\pi^{-s} \Gamma(s) \zeta_{id}(j,s) = \pi^{s-\frac{j}{2}} \Gamma\left(\frac{j}{2} - s\right) \zeta_{id}\left(j,\frac{j}{2} - s\right),
\end{align*}
which is due to Epstein \cite{Epstein}.   It can be analytically continued to the whole complex plane except for a simple pole at $s = \frac{j}{2}$.
This corresponds to \eqref{modified functional equation for phi and psi} for $a(n) = b(n) = r_{j}(n),  c = \frac{1}{\pi}, A=1, B=0, \delta = \frac{j}{2}$ and $\nu = 1.$ 
When $j=2$, it has been conjectured that all the non-trivial zeros of $\zeta_{id}(2,s)$ lies on the critical line $\Re(s) = \frac{1}{2}.$
\end{example}
Exploring whether a given subset $\mathcal{A}$ of natural numbers $\mathbb{N}$ can be used to express any natural number $n$ as a sum of elements from $\mathcal{A}$, and if so, quantifying the various ways this can be achieved, is always a compelling topic of study. In the context of Example \ref{r_b(n)}, the subset $\mathcal{A}$ is defined as the set of all perfect squares. 
%We have exact formulas for $r_2(n)$ and $r_4(n)$ as in \cite[Theorem 3.2.1. and Theorem 3.3.1.]{Partition book}. 
Typically, Epstein zeta functions are investigated for positive definite $j\times j$ matrix $M$ and are defined as
\begin{align*}
\zeta_{M}(j,s) = \sum_{\kappa \in \mathbb{Z}^j/ \{0\}} (\kappa^T M \kappa)^{-s} \quad {\rm{for}} \quad \Re(s) > \frac{j}{2}.
\end{align*}
Example \ref{r_b(n)} deals with $M$ being identity matrix. For $j = 2,4,6$ and $8$, the Epstein zeta function can be expressed in terms of Riemann zeta function and Dirichlet $L$-functions.  An elegant formula for $\zeta_{id}(2,s)$ can be found in \cite[(1.1)]{Zucker}, which states that
\begin{align}\label{Epstein in terms of zeta}
\zeta_{id}(2,s) = 4 \zeta(s) \beta(s),
\end{align}
where $\beta(s) = \sum_{k=1}^{\infty}
\frac{(-1)^{k-1}}{(2k-1)^s}$, 
which is noting but $L(s,  \chi_{4,2})$. 
%Based on numerical evidence in \cite{kawalec} , it has been conjectured that all the non-trivial zeros of $\beta(s)$ must lie on the critical line $\Re(s) = \frac{1}{2}$. If this conjecture were to hold true, it would imply that the Epstein Riemann hypothesis for $j=2$, that is,  
The generalized Riemann hypothesis for $L(s,  \chi_{4,2})$ says that all the non-trivial zeros will lie on $\Re(s)=1/2$.  Thus,  
all the non-trivial zeros of $\zeta_{id}(2,s)$ will also lie on $\Re(s) = \frac{1}{2}$. However,  it has been shown that Epstein zeta functions $\zeta_{id}(j,s)$, for $j \geq 3$, do not satisfy the Riemann hypothesis.  A detailed explanation for $j=4$ can be found in \cite[(1.4)-(1.6)]{Trvan} and numerical evidences from \cite[Section 5]{Trvan} ensures the failure of Epstein Riemann hypothesis for higher dimensions.\\
The next example is a particular case of the well known Dedekind zeta function for the imaginary quadratic field $\mathbb{Q}(\sqrt{-D})$.
\begin{example}\label{example dedekind}
Let $K = \mathbb{Q}(\sqrt{-D})$ be an imaginary quadratic field with discriminant $D$.  Then the Dedekind zeta function is defined as 
\begin{align}\label{def dedekind}
\zeta_K(s) := \sum_{n=1}^{\infty} \frac{\mathsf{a}_n}{n^s} \quad {\rm{for}} \quad \Re(s) >1,
\end{align}
where $\mathsf{a}_n$ denotes the number of integral ideals of norm $n$ of an imaginary quadratic field $\mathbb{Q}(\sqrt{-D})$.  It satisfies the following functional equation 
\begin{align}\label{fn eq dedekind}
\left(\frac{2\pi}{\sqrt{D}}\right)^{-s} \Gamma(s) \zeta_K(s) = \left(\frac{2\pi}{\sqrt{D}}\right)^{-(1-s)} \Gamma(1-s) \zeta_K(1-s),
\end{align}
due to \cite[p.~266]{Overholt}. This fits in \eqref{modified functional equation for phi and psi} for $c= \frac{\sqrt{D}}{2\pi}, A=1, B=0, \delta = 1$ and $\nu = 1$.
The extended Riemann hypothesis asserts that all the non-trivial zeros of $\zeta_K(s)$ lie on $\Re(s) = \frac{1}{2}$.
\end{example}
Before stating our main theorems,  we need to define the Dirichlet inverse of an arithmetical function $a(n)$.  Let $a(n)$ be an arithmetical function with $a(1) \neq 0$, then there exist a unique arithmetical function $a^{-1}(n)$, known as Dirichlet inverse of $a(n)$, such that
\begin{align*}
a*a^{-1} = a^{-1}*a = I.
\end{align*} 
Furthermore, $a^{-1}(n)$ is recursively defined as
\begin{align*}
a^{-1}(1) = \frac{1}{a(1)} ,\hspace{5mm} a^{-1}(n) = \frac{-1}{a(1)} \sum_{\substack{d|n \\ d<n}} a\left(\frac{n}{d}\right)a^{-1}(d) \quad {\rm{for}} \quad n>1.
\end{align*}
A proof of this standard result can be found in \cite[p.~30]{apostal}. 
%\end{theorem}
%\begin{proof}
%Proof of this result can be found in \cite[p.~30]{apostal}
%\end{proof}
Now we are ready to state our main results.

\section{Main Results}
In this section,  we state the main results of this paper.  

\begin{theorem}\label{General_Identity}
Let $a(n)$ and $b(n)$ be two arithmetical functions whose corresponding $L$-functions are $\phi(s)$ and $\psi(s)$ and satisfy the functional equation \eqref{modified functional equation for phi and psi}. Assume that all the non-trivial zeros of $\phi(s)$ are simple. Then for any $k \geq \delta + \frac{B}{A}$ and  $x>0$, we have
\begin{align}\label{main identity}
& \sum_{n=1}^{\infty} \frac{a^{-1}(n)}{n^k} \exp \left({-\frac{ x}{n^{1/A}}}\right) \nonumber \\
& = \frac{c^{-\frac{A\delta+2B}{A}}}{\nu x^{Ak+B}} \frac{\Gamma(Ak+B)}{\Gamma(A\delta +2B)}  \sum_{n=1}^{\infty}\frac{b^{-1}(n)}{n^{\frac{A\delta + B}{A}}} \bigg[  {}_1F_{1} \left( \begin{matrix}
Ak+B \\ A\delta+2B  \end{matrix}  \Big| -\frac{1}{(c^2 n)^{1/A}x} \right) - 1 \bigg]  \nonumber \\
&  -R_{t} + A \sum_{\rho} \frac{  \Gamma(A(k-\rho))   }{\phi'(\rho)} x^{-A(k-\rho)},
\end{align}
where $\rho$ runs through the non-trivial zeros of $\phi(s)$ and the convergence of the infinite series over $\rho$ adhere to the bracketing condition \eqref{bracketing}, and $R_{t}$ represents the residual term corresponding to the trivial zeros of $\phi(k- \frac{s}{A})$ with $0\leq \Re\left(  k- \frac{s}{A} \right) \leq Ak+B$. 

\end{theorem}
%\begin{remark}
%When $\phi(s)$ satisfy the Riemann hypothesis, then we can explicitly evaluate the residual term $R_\rho$ as follows
%\begin{align*}
%R_{\rho} = -  \frac{ A \Gamma(A(k-\rho))   }{\phi'(\rho)} x^{-A(k-\rho)}.
%\end{align*}
%\end{remark}
\begin{remark}
Setting $a(n) = b(n) = 1$ and thus $a^{-1}(n)= b^{-1}(n) = \mu(n)$, $\phi(s) = \psi(s) = \zeta(s)$  in Theorem \ref{General_Identity}   with $A = \frac{1}{2}, B = 0, c = \frac{1}{\sqrt{\pi}}, \delta = 1 $ and $\nu = 1,$ one can easily obtain the identity \eqref{AGM}.
\end{remark}
\begin{remark}
Let $\chi$ be a primitive Dirichlet character modulo $q$ and $L(s,\chi)$ be the associated Dirichlet series.  Letting $a(n) = \chi(n), b(n) = \overline{\chi(n)}$ and thus $a^{-1}(n)= \chi(n)\mu(n),  b^{-1}(n) = \overline{\chi(n)} \mu(n) $,   $\phi(s) = L(s, \chi), \psi(s) = L(s, \overline{\chi}  )$ in Theorem \ref{General_Identity} with $A = \frac{1}{2}, B = \frac{a}{2}, c = \sqrt{\frac{q}{\pi}}, \delta = 1 $ and $\nu = \frac{G(\chi)}{i^a \sqrt{q}},$ and upon simplification,  we can  derive \eqref{character_analogue}.
\end{remark}
Now we shall state a few more interesting results coming from our Theorem \ref{General_Identity}.  
%Substituting $k = \delta + \frac{B}{A}$ and $x=\beta >0$ with $\alpha \beta = \frac{1}{c^{2/A}}$ in Theorem \ref{General_Identity},   
First,  we state an ingenious modular relation,  which is a perfect generalization of Ramanujan-Hardy-Littlewood identity \eqref{Hardy-Littlewood symmetric}. 
\begin{corollary}  \label{general_alpha_beta_identity} 
Assuming all the hypothesis of the Theorem \ref{General_Identity} to be true.  Let $\alpha$,  $\beta$ be positive real numbers such that $\alpha \beta = \frac{1}{c^{2/A}}$.  Then we have
\begin{align*}
& \sqrt{\bar{\nu}}\alpha^{\frac{A\delta +2B}{2}} \sum_{n=1}^{\infty}\frac{b^{-1}(n)}{n^{\delta + B/A}}  \exp\left(-\frac{\alpha}{n^{1/A}}\right) - \sqrt{\nu} \beta^{\frac{A\delta +2B}{2}} \sum_{n=1}^{\infty}\frac{a^{-1}(n)}{n^{\delta + B/A}} \exp\left(-\frac{\beta}{n^{1/A} }\right) \nonumber \\ &= \frac{\alpha^{\frac{A\delta +2B}{2}}}{\sqrt{\nu}} \sum_{n=1}^{\infty}\frac{b^{-1}(n)}{n^{\delta + B/A}} + \sqrt{\nu} \beta^{\frac{A\delta +2B}{2}} R_t -  A \sqrt{\nu} \sum_{\rho}  \frac{  \Gamma(A(\delta-\rho)+B)   }{\phi'(\rho)} \beta^{ \frac{A \delta-2B}{2}}  .
\end{align*}
\end{corollary}

The next result is corresponding to the Epstein zeta function,  defined in Example \ref{r_b(n)}.  
\begin{corollary}\label{corolary for r_2(n)}
Let $k \geq 1$ and $x$ be any positive real number.  Let $r_2(n)$ be the number of ways $n$ can be written as sum of two squares and the associated $L$-function is  defined in Example \ref{r_b(n)} for $j =2.$ Then, we have 
\begin{align}\label{main identity for r_2(n)}
\sum_{n=1}^{\infty} \frac{r_2^{-1}(n)}{n^k}  \exp\left(-\frac{x}{n}\right) = \frac{\pi \Gamma(k)}{x^k} \sum_{n=1}^{\infty} \frac{r_2^{-1}(n)}{n}  {}_1F_{1} \left( \begin{matrix}
k \\ 1  \end{matrix}  \Big|- \frac{ \pi^2}{nx} \right)  +  \sum_{\rho} \frac{\Gamma(k-\rho)}{\zeta_{id}^{'}(2, \rho)} x^{-(k-\rho)}.
\end{align}
Here, $\rho$ runs over all the non-trivial zeros of $\zeta_{id}(2,s)$ and are assumed to be simple. The convergence of the series over $\rho$ follows under the assumption of the bracketing condition \eqref{bracketing}.
\end{corollary}
A few numerical evidences for the above result is given in Table \ref{Table for r_2(n)}.  
\begin{remark}
In particular,  letting $k=1$ and $x= \beta>0$ with $\alpha \beta = \pi^2$ in Corollary \ref{corolary for r_2(n)}, we obtain the following beautiful identity,  
\begin{align*}
\sqrt{\alpha}\sum_{n=1}^{\infty} \frac{r_2^{-1}(n)}{n} \exp\left(-\frac{\alpha}{n}\right) &- \sqrt{\beta}\sum_{n=1}^{\infty} \frac{r_2^{-1}(n)}{n} \exp\left(-\frac{\beta}{n}\right)=  - \frac{1}{\sqrt{\beta}} \sum_{\rho} \frac{\Gamma(1-\rho)}{\zeta_{id}^{'}(2, \rho)} \beta^{\rho}.
\end{align*}
\end{remark}
The next identity is associated to the Dedekind zeta function over an imaginary quadratic field.  
\begin{corollary}\label{corollary for dedeking zeta}
Let $k \geq 1$ be any real number. If $\mathsf{a}_n$ counts the number of integral ideals of norm $n$ of an imaginary quadratic field $\mathbb{Q}(\sqrt{-D})$ and the associated $L$-function is defined as in \eqref{def dedekind} and satisfies the functional equation \eqref{fn eq dedekind}, then for any $x>0$, we have
\begin{align*}
\sum_{n=1}^{\infty} \frac{\mathsf{a}_n^{-1}}{n^k} \exp\left(-\frac{x}{n}\right)&= 	   \frac{2\pi \Gamma(k)}{\sqrt{D} x^k} \sum_{n=1}^{\infty} \frac{\mathsf{a}_n^{-1}}{n}  {}_1F_{1} \left( \begin{matrix}
k \\ 1  \end{matrix}  \Big|- \frac{ 4\pi^2}{Dnx} \right)  +  \sum_{\rho} \frac{\Gamma(k-\rho)}{\zeta_K^{'}(\rho)} x^{-(k-\rho)},
\end{align*}
where $\rho$ runs over the non-trivial zeros of $\zeta_K(s)$ and convergence of the infinite series over $\rho$ follows from the bracketing condition \eqref{bracketing}.
\end{corollary}
%{\bf Add $k=1$ case with $\alpha $ and $\beta$.}
\begin{remark} In particular, substituting $k=1$ and $x=\beta>0$ and $\alpha \beta = \frac{4\pi^2}{D}$ in Corollary \ref{corollary for dedeking zeta}  gives the following identity\footnote{In \cite[Cor. ~3.4]{DGV21},  there is an extra $1/2$ on the right side of the identity. } of Dixit, Gupta and Vatwani \cite[Cor. ~3.4]{DGV21},
\begin{align*}
\sqrt{\alpha} \sum_{n=1}^{\infty} \frac{\mathsf{a}_n^{-1}}{n} \exp\left(-\frac{\alpha}{n}\right)-\sqrt{\beta} \sum_{n=1}^{\infty} \frac{\mathsf{a}_n^{-1}}{n} \exp\left(-\frac{\beta}{n}\right)= - \frac{1}{\sqrt{\beta}}  \sum_{\rho} \frac{\Gamma(1-\rho)}{\zeta_K^{'}(\rho)} \beta^{\rho}.
\end{align*}
\end{remark}

The next result corresponds to $\phi(s) = \zeta(s) \zeta(s-r)$,  defined in Example \ref{def zeta(s) zeta(s-r)},  an important example of a Dirichlet series lying in Chadrasekharan-Narasimhan class of $L$-functions,  which lacks the Riemann hypothesis.
\begin{theorem}\label{corollary for zeta(s) zeta(s-r)}
 Let $\sigma_r(n)$ denote the sum of $r^{th}$ powers of positive divisors of $n$. Then the corresponding $L$-function as defined in Example \ref{def zeta(s) zeta(s-r)} satisfies the functional equation \eqref{fn equ zeta(s) zeta(s-r)}. Let $k \geq r+1$, where $r$ is an odd positive integer.  Then for any positive $x$, we have
\begin{align*}
\sum_{n=1}^{\infty} \frac{\sigma_r^{-1}(n)}{n^k} \exp\left(-\frac{x}{n}\right)
&=  \frac{(-1)^{\frac{r+1}{2}}(2\pi)^{r+1}}{x^k} \frac{\Gamma(k)}{\Gamma(r+1)} \sum_{n=1}^{\infty} \frac{\sigma_r^{-1}(n)}{n^{r+1}} {}_1F_{1} \left( \begin{matrix}
k \\ r+1  \end{matrix}  \Big|- \frac{4 \pi^2}{nx} \right)  \nonumber \\ 
& + \sum_{m=1}^{[\frac{r}{2}]} \frac{\Gamma(k-r+2m)}{\zeta(r-2m) \zeta(2m+1) } \frac{(-1)^m 2(2\pi)^{2m}}{(2m)!} x^{r-k-2m}   \nonumber \\ 
& + \sum_{\rho} \frac{x^{\rho-k}}{\zeta^{'}(\rho)} \bigg[\frac{\Gamma(k-\rho)}{\zeta(\rho -r)} + \frac{x^r \Gamma(k-r-\rho)}{\zeta(r+\rho)}\bigg],
\end{align*}
where $\rho$ runs over non-trivial zeros of $\zeta(s)$ and are assumed to be simple. The convergence of the associated series follows under the assumption of bracketing condition \eqref{bracketing}.
\end{theorem}
A numerical verification of the identity is given in Table \ref{Table of main theorem}.

In particular,  letting $k=r+1$ and $\alpha \beta = 4 \pi^2$ in Theorem \ref{corollary for zeta(s) zeta(s-r)},  we get the following identity.  
\begin{corollary}\label{alpha-beta form for sigma}
Let $r$ be an odd positive integer and $\alpha,  \beta$ be positive real numbers with $\alpha \beta = 4 \pi^2$.  Then we have
\begin{align*}
\alpha^{\frac{r+1}{2}}&\sum_{n=1}^{\infty} \frac{\sigma_r^{-1}(n)}{n^{r+1}} \exp\left(-\frac{\alpha}{n}\right) - (-\beta)^{\frac{r+1}{2}}\sum_{n=1}^{\infty} \frac{\sigma_r^{-1}(n)}{n^{r+1}} \exp\left(-\frac{\beta}{n}\right) \nonumber \\ 
&=   2  \alpha^{\frac{r-1}{2}} \sum_{m=1}^{[\frac{r}{2}]} \frac{(-1)^m}{\zeta(r-2m) \zeta(2m+1) } \left( \frac{\beta}{2\pi}  \right)^{2m} \nonumber \\ 
& + \alpha^{\frac{r+1}{2}-k} \sum_{\rho} \frac{\alpha^{\rho}}{\zeta^{'}(\rho)} \bigg[\frac{\Gamma(r+1-\rho)}{\zeta(\rho -r)} + \frac{\alpha^r \Gamma(1-\rho)}{\zeta(r+\rho)}\bigg]. 
\end{align*} 
\end{corollary}

The next result, for $L$-functions associated to cusp forms,  we state as a separate theorem as it might have a significant importance in the literature.  

\begin{theorem}\label{corollary for cusp form}
Let $L(f,s)$ be the $L$-function associated to a cusp form $f$ defined in \eqref{cusp form}. If we assume all the non-trivial zeros of $L(f,s)$ to be simple, then for $k \geq \frac{w+1}{2}$,  $x>0$, we have
\begin{align}
\sum_{n=1}^{\infty} \frac{\lambda_f^{-1}(n)}{n^k} \exp\left(-\frac{x}{n}\right)&= \frac{(2\pi)^{\omega}}{ x^{k + \frac{\omega -1}{2}}} \frac{\Gamma(k + \frac{\omega -1}{2})}{\Gamma(\omega)} \sum_{n=1}^{\infty} \frac{\lambda_f^{-1}(n)}{n^{\frac{\omega + 1}{2}}} {}_1F_{1} \left( \begin{matrix}
 k + \frac{\omega -1}{2}\\ \omega \end{matrix}  \Big|- \frac{4 \pi^2}{nx} \right)  \nonumber \\ 
 &+  \sum_{\rho} \frac{\Gamma(k-\rho) x^{-(k-\rho)}}{L^{'}(f,\rho)}, \label{main identity for cusp form}
\end{align}
where $\rho$ runs through non-trivial zeros of $L(f,s)$ and satisfy the bracketing condition \eqref{bracketing}.
\end{theorem}

%{\bf Add $k= (w+1)/2$ with $\alpha$ and $\beta$.}

By substituting $k=\frac{\omega + 1}{2}$ in Theorem \ref{corollary for cusp form},  we obtain an elegant modular relation depicted as follows. 
\begin{corollary}\label{alpha beta form for cusp forms}
 Assuming all the hypothesis in Theorem \ref{corollary for cusp form},  for $\alpha, \beta > 0$ with $\alpha \beta = 4\pi^2$,  we have
\begin{align*}
\alpha^{\frac{\omega}{2}}\sum_{n=1}^{\infty} \frac{\lambda_f^{-1}(n)}{n^{\frac{\omega + 1}{2}}} \exp\left(-\frac{\alpha}{n}\right)&-\beta^{\frac{\omega}{2}}\sum_{n=1}^{\infty} \frac{\lambda_f^{-1}(n)}{n^{\frac{\omega + 1}{2}}} \exp\left(-\frac{\beta}{n}\right) = - \frac{1}{\sqrt{\beta}}  \sum_{\rho} \frac{\Gamma(\frac{\omega + 1}{2}-\rho) \beta^\rho}{L^{'}(f,\rho)}.
\end{align*}
%Here, $\rho$ runs through non-trivial zeros of $L(f,s)$ and convergence of the series on right hand side over $\rho$ follows if we assume the bracketing condition \eqref{bracketing}.
\end{corollary}

%\begin{corollary}\label{corollary for Ramanujan tau function}
%Let $\tau_0(n)$ be as defined in Example \ref{def tau }. Then for $k \geq 1$ and $x$ be any positive real number, we have
%\begin{align}\label{main identity for Ramanujan tau function}
%\sum_{n=1}^{\infty} \frac{\tau_0^{-1}(n)}{n^k} \exp\left(-\frac{x}{n}\right)&= \frac{(2\pi)^{12}\Gamma\left(k + \frac{11}{2}\right)}{11! x^{k+\frac{11}{2}}} \sum_{n=1}^{\infty} \frac{\tau_0^{-1}(n)}{n^{\frac{13}{2}}}  {}_1F_{1} \left( \begin{matrix}
%k+\frac{11}{2}\\ 12 \end{matrix}  \Big|- \frac{4 \pi^2}{nx} \right)  \nonumber \\ 
%& + \sum_{\rho} \frac{\Gamma(k-\rho) x^{-(k-\rho)}}{L^{'}( \Delta , \rho,)},
%\end{align}
%where $\rho$ runs over non-trivial zeros of $L(\Delta,s )$.
%\end{corollary}

The next corollary gives a modular relation involving Ramanujan tau function.
\begin{corollary}\label{corollary for Ramanujan tau function}
Let $\alpha$ and $\beta$ be two positive real numbers with $\alpha \beta = 4\pi^2$. Then we have 
\begin{align}\label{identity for Ramanujan tau}
 \alpha^6 \sum_{n=1}^{\infty} \frac{\tau_0^{-1}(n)}{n^{\frac{13}{2}}}  \exp\left(-\frac{\alpha}{n}\right) &-  \beta^6 \sum_{n=1}^{\infty} \frac{\tau_0^{-1}(n)}{n^{\frac{13}{2}}}  \exp\left(-\frac{\beta}{n}\right) = - \frac{1}{\sqrt{\beta}} \sum_{\rho} \frac{\Gamma(\frac{13}{2}-\rho) \beta^{\rho}}{L^{'}( \Delta , \rho)},
\end{align}
 where $\rho$ runs through non-trivial zeros of $L(\Delta,s)$.  
\end{corollary}
It is interesting to note that the above identity has already been obtained by Dixit,  Roy and Zaharescu \cite[Corollary 1.2]{DRZ}.  One can also obtain Corollary \ref{alpha beta form for cusp forms}  by letting $z \rightarrow 0$ in Theorem 1.1 of \cite{DRZ}.  However,  our Theorem \ref{corollary for cusp form} is different from Theorem 1.1 of Dixit et.  al \cite{DRZ}.  

\subsection{Equivalent criteria for the grand Riemann hypothesis}
Motivated from the equivalent criteria for the Riemann hypothesis of $\zeta(s)$ due to Hardy-Littlewood and Riesz, 
in this subsection,  we state equivalent criteria for the Riemann hypothesis for a general class of $L$-functions,  defined in \eqref{def nyc L function} and satisfying Littlewood's bound \eqref{merten bound general}.  In particular,  we give equivalent criteria for the Riemann hypothesis for $L$-functions lying in the Chadrasekharan-Narasimhan class that may obey Riemann hypothesis.  

\begin{theorem}\label{conjecture is true}
Conjecture \ref{conjecture} holds true.
\end{theorem}
As we know that the  Epstein zeta function $\zeta_{id}(2,s)$ and $L$-functions $L(f,s)$ associated to Hecke eigenforms are particular examples of nice $L$-functions that are lying in the Chadrasekharan-Narasimhan class.
Thus,  as an immediate implication,  we have the following results.  
\begin{corollary}\label{equivalent criteria for RH  for epstein}
Let $\zeta_{id}(2,s)$ be the Epstein zeta function defined in Example \ref{r_b(n)}.  For $k \geq 1,  \ell > 0$,  the Riemann hypothesis for $\zeta_{id}(2,s)$ is equivalent to 
\begin{align*}
\sum_{n=1}^{\infty} \frac{r_2^{-1}(n)}{n^k} \exp\left(-\frac{x}{n^{\ell}}\right) = O_{\epsilon, k, \ell}\left(x^{-\frac{k}{\ell} + \frac{1}{2\ell}+\epsilon}\right),
\end{align*}
for any $\epsilon > 0$, as $x \rightarrow \infty$.
\end{corollary}
\begin{corollary}\label{equivalent criteria for RH  for hecke}
Let $f$ be a holomorphic Hecke eigenform of weight $\omega$ over the full modular group $SL_2(\mathbb{Z})$ and $L(f,s)$ be the $L$-function defined in \eqref{cusp form}.  For $k \geq \frac{\omega+1}{2},  \ell >0$,   the Riemann hypothesis for $L(f,s)$ is equivalent to the following bound:
\begin{align*}
\sum_{n=1}^{\infty} \frac{\lambda_f^{-1}(n)}{n^k} \exp\left(-\frac{x}{n^{\ell}}\right) = O_{\epsilon, k,f, \ell}\left(x^{-\frac{k}{\ell} + \frac{1}{2\ell}+\epsilon}\right),
\end{align*}
for any $\epsilon > 0$ as $x \rightarrow \infty$.
\end{corollary}
In the next result, we state a new equivalent criteria for the Riemann hypothesis of $\zeta(s)$. This criteria significantly differ  from the previous known equivalent criteria for RH due to Hardy-Littlewood and Riesz.  The heuristic of this bound comes from the identity mentioned in Theorem \ref{corollary for zeta(s) zeta(s-r)}.

\begin{theorem}\label{Equivalent criteria for RH by sigma}
Let $r$ be a non-negative integer and $k \geq r+1,  \ell >0$.   The Riemann hypothesis for $\zeta(s)$ is equivalent to the bound 
\begin{align}\label{bound for eq criteria for sigma}
P_{k,\sigma,\ell}(x) := \sum_{n=1}^{\infty}\frac{\sigma_r^{-1}(n)}{n^k} \exp\left(-\frac{x}{n^{\ell}}\right) = O_{\epsilon, k,r, \ell}\left(x^{-\frac{k}{\ell} +\frac{1+2r}{2 \ell}+\epsilon}\right),
\end{align}
for any $\epsilon > 0$, as $x \rightarrow \infty$.
\end{theorem}

\section{Nuts and Bolts}

In this section,  we state a few well-known results,  which will play a vital role in proving our main identity.  We begin with an important asymptotic formula, namely, stirling's formula which gives information about the asymptotic behaviour of $\Gamma(s)$. 
\begin{lemma}\label{Stirling}
In a vertical strip $r_1 \leq \sigma \leq r_2$,  we have
\begin{equation*}
|\Gamma (\sigma + i T)| = \sqrt{2\pi} | T|^{\sigma - 1/2} e^{-\frac{1}{2} \pi |T|} \left(1 + O\left(\frac{1}{|T|}\right)  \right), \quad {\rm as} \quad |T|\rightarrow \infty.
\end{equation*}
\end{lemma}
Next, we document Euler's summation formula, which will be useful in deriving the equivalent criteria for the Riemann hypothesis of the $L$-functions.  
\begin{lemma} \label{Euler's summation}
Let $\{ a_n\}$ be a sequence of complex numbers and $f(t)$ be a continuously differentiable function on $[1,x]$.  Consider $A(x):= \sum_{1 \leq n \leq x} a_n$.  Then we have
\begin{align*}
\sum_{ 1\leq n \leq x} a_n f(n) = A(x) f(x) - \int_{1}^{x} A(t) f'(t) {\rm d}t.
\end{align*}
\end{lemma}
\begin{proof}
A proof of this result can be found in \cite[p.~17]{Murty}.  
\end{proof}
The next result gives an information about the behaviour of general $L$-functions in a vertical strip.  
\begin{lemma}\label{L function bound to vanish horizontal integral}
Any $L$-function is polynomially bounded in vertical strips $\sigma_1 < \Re(s) < \sigma_2,$ and $\Im(s) \geq 1$.
\end{lemma}
\begin{proof}
We refer \cite[Lemma 5.2]{IK} for the proof.
\end{proof}
 The forthcoming result will serve as one of the vital components in establishing equivalent criteria for the grand Riemann hypothesis. 
 \begin{lemma}\label{main lemma for converse part}
Let $L(f,s)$ be a nice $L$-function and its inverse defined in \eqref{def nyc L function} and \eqref{inverse of L(f,s)},  respectively.  For $k \geq 1$, $\ell >0$,  we define the function
\begin{align*}
P_{k,f,\ell}(x): = \sum_{n=1}^{\infty}\frac{\mu_f(n)}{n^{k}} \exp \left(- \frac{ x}{n^{\ell}}\right). 
\end{align*}
Then for $\Re(s) > \frac{1-k}{\ell}$, except for  $s = 0,1,2 \cdots$, we have
\begin{align*}
 \int_{0}^{\infty}x^{-s-1}P_{k,f,\ell}(x) \mathrm{d}x = \frac{\Gamma(-s)}{L(f, \ell s+k)}.
\end{align*}
\end{lemma}
\begin{proof}
We start with the definition of $P_{k,f, \ell}(x)$. For $k \geq 1$,  we can write
\begin{align}\label{P_k,f,l}
P_{k,f,\ell}(x) = \sum_{n=1}^{\infty}\frac{\mu_f(n)}{n^{k}} \exp \left(- \frac{ x}{n^{\ell}}\right)  & = \sum_{n=1}^{\infty}\frac{\mu_f(n)}{n^{k}} \sum_{m=0}^{\infty}\frac{(-1)^m x^m}{n^{\ell m} m!} \nonumber \\ & = \sum_{m=0}^{\infty}\frac{(-1)^m x^m}{ m!}\sum_{n=1}^{\infty}\frac{\mu_f(n)}{n^{k+\ell m}} \nonumber \\ & = \sum_{m=0}^{\infty}\frac{(-1)^m x^m}{ m!L(f, \ell m+k)} \nonumber \\&= \frac{1}{L(f,k)} + \sum_{m=1}^{\infty}\frac{(-1)^m x^m}{ m!L(f, \ell m+k)}.
\end{align}
%Here interchange of summation and integration is justified as $\ell m+k>1$. Therefore, one has
%\begin{align*}
%\sum_{m=0}^{\infty}  \sum_{n=1}^{\infty} \left| \frac{\mu_f(n)}{n^{k+\ell m}} \frac{(-1)^m x^m}{ m!} \right| \leq \zeta(k-\epsilon)e^x < \infty.
%\end{align*}
%Note that $\frac{1}{L(f,k)}$ vanishes for $k=1$ if $L(f,s)$ has a pole at $s=1$. 
For $\Re(s) > \frac{1-k}{\ell}$, use the series expansion \eqref{def nyc L function} of $L(f, k+\ell s)$ to write
\begin{align}\label{final_intergral}
L(f, k+\ell s) \int_{0}^{\infty}x^{-s-1}P_{k,f,\ell}(x) \mathrm{d}x &=  \sum_{n=1}^{\infty}\frac{A_f(n)}{n^k} \int_{0}^{\infty}\frac{x^{-s-1}}{n^{\ell s}} P_{k,f,\ell}(x) \mathrm{d}x \nonumber \\ & =  \sum_{n=1}^{\infty}\frac{A_f(n)}{n^k} \int_{0}^{\infty}x^{-s-1} P_{k,f,\ell}\left(\frac{x}{n^{\ell}}\right)\mathrm{d}x \nonumber \\ &= \int_{0}^{\infty}x^{-s-1}\sum_{n=1}^{\infty}\frac{A_f(n)}{n^k} P_{k,f,\ell}\left(\frac{x}{n^{\ell}}\right) \mathrm{d}x.
\end{align}
Here, in the second last step, we changed the variable $x$ by $\frac{x}{n^{\ell}}$ and interchange of the order of summation and integration in the last step is facilitated due to dominated convergence theorem. This is because the infinite sum is bounded by $e^{-x}$,  see \eqref{infinite sum boubded by exp},  which is integrable in the interval $(0,\infty)$. Now, utilize \eqref{P_k,f,l} to obtain
\begin{align}\label{infinite sum boubded by exp}
\sum_{n=1}^{\infty}\frac{A_f(n)}{n^k} P_{k,f,\ell}\left(\frac{x}{n^{\ell}}\right) &= \sum_{n=1}^{\infty}\frac{A_f(n)}{n^k}\left(\frac{1}{L(f,k)} + \sum_{m=1}^{\infty}\frac{(-1)^m x^m}{ m! L(f, k+\ell m)n^{\ell m}}\right) 
\nonumber \\ &=\begin{cases}
e^{-x} -1,  & \text{if}~ k=1~\text{and}~L(f,s)~ \text{has a pole at}~s=1, \\
e^{-x},  & \text{otherwise}. 
\end{cases}
\end{align}
In this context, combining \eqref{infinite sum boubded by exp} and \eqref{final_intergral},  when $k=1$ and $L(f,s)$ has a pole at $s=1$,  we obtain the following expression
\begin{align*}
L(f, k+\ell s)\int_{0}^{\infty}x^{-s-1}P_{k,f,\ell}(x) \mathrm{d}x &=  \int_{0}^{\infty}x^{-s-1} (e^{-x}-1) \mathrm{d}x \\&= \frac{-1}{s} \int_{0}^{\infty} x^{(1-s)-1} e^{-x} \mathrm{d}x \
\\ &= - \frac{\Gamma(1-s)}{s}.
\end{align*}
This result is valid for $0< \Re(s)<1$.  Note that we used the fact that $\lim_{x \rightarrow 0} \frac{e^{-x}-1}{x^s} = 0$ for $0< \Re(s)<1$. Furthermore,  we know $- \frac{\Gamma(1-s)}{s} = \Gamma(-s)$.  
In other case,  we can see that
\begin{align*}
L(f, k+\ell s)\int_{0}^{\infty}x^{-s-1}P_{k,f,\ell}(x) \mathrm{d}x =  \int_{0}^{\infty}x^{-s-1} e^{-x} \mathrm{d}x = \Gamma(-s),
\end{align*}
valid for $\frac{1-k}{\ell} < \Re(s) <0$. By analytic continuation, we can extend this to right half plane $\Re(s) > \frac{1-k}{\ell}$ except $s = 0,1,2 \cdots$. The proof is now complete.
\end{proof}
\begin{lemma}\label{main lemma for sigma bound}
Let $r$ be a non-negative integer and  $k \geq r+1$ be any positive real number.  We define
\begin{align*}
P_{k,  \sigma,\ell}(x): = \sum_{n=1}^{\infty}\frac{\sigma_{r}^{-1}(n)}{n^{k}} \exp \left(- \frac{ x}{n^{\ell}}\right). 
\end{align*}
Then for $Re(s) > \frac{1-k+r}{\ell}$ except $s = 0,1,2 \cdots$, we have

\begin{align*}
\int_{0}^{\infty}x^{-s-1}P_{k,\sigma,\ell}(x) \mathrm{d}x = \frac{\Gamma(-s)}{\zeta(k+\ell s)\zeta(k+\ell s-r)}.
\end{align*}
\end{lemma}
\begin{proof}
The proof follows along the same line as of Lemma \ref{main lemma for converse part}. Thus, we skip the details here.  
\end{proof}
Now we state a result for the growth of the summatory function of an arithmetical function.  
\begin{lemma}\label{for_mertens_bound}
Let $a(n)$ be an arithmetical function such that $a(n) = O(\Phi(n))$,  where $\Phi(n)$ is an increasing function for $x \geq x_0$. Suppose that the Dirichlet series $A(s) = \sum_{n=1}^{\infty}\frac{a(n)}{n^s}$ is absolutely convergent for $\Re(s) > c_0$ and 
\begin{align*}
\sum_{n=1}^{\infty} \frac{a(n)}{n^{\sigma}} \leq \frac{1}{(\sigma - c_0)^{\gamma}} \quad \rm{as} \quad \sigma \rightarrow c_0^{+},
\end{align*}
for some $\gamma> 0$. Then for any $c > c_0$, one has 
\begin{align*}
\sum_{n \leq x}a(n) &= \frac{1}{2 \pi i} \int_{c-iT}^{c+iT} A(w)x^w \frac{\mathrm{d}w}{w} + O\left(\frac{x^c}{T(c-c_0)^{\gamma}}\right) \nonumber \\   &+ O\left(\frac{\Phi(2x) x \log(2x)}{T}\right) + O\left(\Phi(2x)\right),
\end{align*}
where $T$ is some large positive real number. 
\end{lemma}
\begin{proof}
This result can be found in \cite[p.~486,  Equation (A.10)]{Ivic}.
\end{proof}
Using the above identity,  under the assumption of the Riemann hypothesis for $\zeta(s)$,  we obtain the below result which is an analogue of  Littlewood's bound \eqref{merten bound general}.   

\begin{lemma}\label{mertens bound for sigma} Let $r$ be a non-negative integer.  
Under the assumption of the Riemann hypothesis for $\zeta(s)$,  we have
\begin{align*}
\sum_{n \leq x} \sigma_{r}^{-1}(n) = O\left(x^{\frac{1}{2}+ r + \epsilon}\right).
\end{align*}
%{\bf Can we prove the converse part?}
\end{lemma}

\begin{proof}
First,  assuming RH for the Riemann zeta function,  we will prove this bound.  In Lemma \ref{for_mertens_bound},  we consider $a(n)=  \sigma_{r}^{-1}(n) = \mu N^r * \mu$. Then
\begin{align*}
|a(n)| = \bigg|\sum_{d|n}\mu(d)d^r \mu\left(\frac{n}{d}\right)\bigg| \leq \sum_{d|n} d^r \leq \sum_{d|n} n^r = n^r d(n)= O(n^{r+\epsilon}),
\end{align*}
for any $\epsilon >0$.  
This suggests us to take $\Phi(n) = n^{r+\epsilon}$. Thus, with $A(s) = \frac{1}{\zeta(s) \zeta(s-r)}$ and $c = 1+ r + \epsilon$, we see that
\begin{align*}
\sum_{n \leq x} \sigma_{r}^{-1}(n) &= \frac{1}{2\pi i}\int_{1+r+\epsilon-iT}^{1+r+\epsilon+iT} \frac{1}{\zeta(w)\zeta(w-r)}\frac{x^w}{w} \mathrm{d}w + O\left(\frac{x^{1+r+\epsilon}}{T}\right) \nonumber \\   & + O\left(\frac{x^{r+\epsilon} x \log(2x)}{T}\right) + O\big(x^{r+\epsilon}\big). 
\end{align*}
Now we change the line of integration from $\Re(s) = 1+r+\epsilon $ to $\Re(s) = \frac{1}{2}  + r + \epsilon$. Then we have
\begin{align}\label{asymptotic for RH}
\sum_{n \leq x} \sigma_{r}^{-1}(n) &= \frac{1}{2\pi i} \bigg[\int_{\frac{1}{2} + r + \epsilon +iT}^{1+r+\epsilon+iT}+\int_{\frac{1}{2}  + r + \epsilon -iT}^{\frac{1}{2}  + r + \epsilon+iT} + \int_{1+r+\epsilon-iT}^{\frac{1}{2}  + r + \epsilon-iT}\bigg]\frac{1}{\zeta(w)\zeta(w-r)}\frac{x^w}{w} \mathrm{d}w   \nonumber\\ & + O\left(\frac{x^{1+r+\epsilon}}{T}\right) + O(x^{r+\epsilon}).
\end{align}
Under the assumption of RH,  we know $\frac{1}{\zeta(\sigma+it)}=O(|t|^\epsilon)$ for $\sigma> 1/2$.  Using this bound and under simplification,   one can check that 
\begin{align*}
\left| \int_{\frac{1}{2} + r + \epsilon +iT}^{1+r+\epsilon+iT}  \frac{1}{\zeta(w)\zeta(w-r)}\frac{x^w}{w} \mathrm{d}w \right|  & \ll \int_{\frac{1}{2} + r + \epsilon }^{1+r+\epsilon} x^{\sigma} T^{2\epsilon-1} d\sigma \ll T^{2\epsilon-1} x^{2+r+\epsilon},  \\
\left| \int_{\frac{1}{2}  + r + \epsilon -iT}^{\frac{1}{2}  + r + \epsilon+iT}  \frac{1}{\zeta(w)\zeta(w-r)}\frac{x^w}{w} \mathrm{d}w \right| & \ll \int_{-T}^{T} |t|^{2\epsilon} x^{\frac{1}{2}+r+\epsilon} \frac{dt}{|t|} \ll T^{2\epsilon} x^{\frac{1}{2}+r+\epsilon},  \\
\left| \int_{1+r+\epsilon-iT}^{\frac{1}{2}  + r + \epsilon-iT}  \frac{1}{\zeta(w)\zeta(w-r)}\frac{x^w}{w} \mathrm{d}w \right| &  \ll  \int_{1+r+\epsilon}^{\frac{1}{2} + r + \epsilon } x^{\sigma} T^{2\epsilon-1} d\sigma \ll T^{2\epsilon-1} x^{2+r+\epsilon}.  
\end{align*}
Utilizing these bounds in \eqref{asymptotic for RH},  we arrive 
\begin{align*}
\sum_{n \leq x} \sigma_{r}^{-1}(n) & =  O\left( T^{2\epsilon-1} x^{2+r+\epsilon} \right) + O\left( T^{2\epsilon} x^{\frac{1}{2}+r+\epsilon} \right)+  O\left(\frac{x^{1+r+\epsilon}}{T}\right) + O\big(x^{r+\epsilon}\big) \\
& =  O\left( T^{2\epsilon-1} x^{2+r+\epsilon} \right) + O\left( T^{2\epsilon} x^{\frac{1}{2}+r+\epsilon} \right).  
\end{align*}
% Equating the above two big oh terms we get $T=x^3/2$. 
Taking $T = x^{\frac{3}{2}}$,  the result follows.  
%\begin{align*}
%\sum_{n \leq x} \sigma_{r}^{-1}(n)= O\left(x^{\frac{1}{2}+ r + \epsilon}\right). 
%\end{align*}

\end{proof}

Now we define an important special function,  namely,  the Meijer $G$-function \cite[p.~415, Definition 16.17]{NIST}, which is a generalization of many well-known special functions.
Let $m,n,p,q$ be non-negative integers with $0\leq m \leq q$, $0\leq n \leq p$. Let $a_1, \cdots, a_p$ and $b_1, \cdots, b_q$ be complex numbers with $a_i - b_j \not\in \mathbb{N}$ for $1 \leq i \leq n$ and $1 \leq j \leq m$. Then the Meijer $G$-function is defined by 
\begin{align}\label{Meijer-G}
G_{p,q}^{\,m,n} \!\left(  \,\begin{matrix} a_1,\cdots , a_p \\ b_1, \cdots , b_q \end{matrix} \; \Big| z   \right) := \frac{1}{2 \pi i} \int_L \frac{\prod_{j=1}^m \Gamma(b_j - s) \prod_{j=1}^n \Gamma(1 - a_j +s) z^s  } {\prod_{j=m+1}^q \Gamma(1 - b_j + s) \prod_{j=n+1}^p \Gamma(a_j - s)}\mathrm{d}s,
\end{align}
where the line of integration $L$, from $-i \infty$ to $+i \infty$,   separates the poles of the factors $\Gamma(1-a_j+s)$  from those of the factors $\Gamma(b_j-s)$.   The above integral converges if $p+q < 2(m+n)$ and $|\arg(z)| < \left(m+n - \frac{p+q}{2} \right) \pi$.

Next,  we state Slater's theorem \cite[p.~415, Equation 16.17.2]{NIST}, which will allow us to express Meijer $G$-function in terms of generalized hypergeometric functions. 
If $p \leq q$ and $ b_j - b_k \not\in \mathbb{Z}$ for $j\neq k$, $1 \leq j, k \leq m$, then 
\begin{align}\label{Slater}
& G_{p,q}^{\,m,n} \!\left(   \,\begin{matrix} a_1, \cdots , a_p \\ b_1, \cdots , b_q \end{matrix} \; \Big| z   \right) \nonumber  \\
& \quad = \sum_{k=1}^{m} A_{p,q,k}^{m,n}(z) {}_p F_{q-1} \left(  \begin{matrix}
1+b_k - a_1,\cdots, 1+ b_k - a_p \\
1+ b_k - b_1, \cdots, *, \cdots, 1 + b_k - b_q 
\end{matrix} \Big| (-1)^{p-m-n} z  \right),  
\end{align}
where $*$ means that the entry $1 + b_k - b_k$ is omitted and 
\begin{align*}
A_{p,q,k}^{m,n}(z) := \frac{ z^{b_k}  \prod_{ j=1,  j\neq k}^{m} \Gamma(b_j - b_k ) \prod_{j=1}^n   \Gamma( 1 + b_k -a_j ) }{ \prod_{j=m+1}^{q} \Gamma(1 + b_k - b_{j}) \prod_{j=n+1}^{p} \Gamma(a_{j} - b_k)  }.
\end{align*}

 \section{Proof of Main Result}
 
\begin{proof}[Theorem \ref{General_Identity}][]
Shifting the line of integration from $\Re(s) > 0$ to $\Re(s)= d \in (-1,0)$ in the inverse mellin transform of $\Gamma(s)$, we have  
\begin{align}\label{inverse mellin}
e^{-x} - 1 = \frac{1}{2 \pi i}  \int_{d-i\infty}^{d+i \infty}\Gamma(s) x^{-s} {\rm d}s. 
\end{align}
We choose $k_0$  such that the following series
$$
\sum_{n = 1}^{\infty} \frac{a^{-1}(n)}{n^k},   \quad \sum_{n=1}^{\infty} \frac{b^{-1}(n)}{n^k} 
$$ 
are convergent for $k=k_0$ and absolutely convergent for all $k > k_0$. 
% In case of the Riemann zeta function and Dirichlet L-function $k_0$ is $1$.  
Then for $k \geq k_0$,  $x>0$, we use \eqref{inverse mellin} to write 
\begin{align}\label{main equation for general CN}
\sum_{n=1}^{\infty} \frac{a^{-1}(n)}{n^k} \exp \left({-\frac{ x}{n^{1/A}}}\right)&=\sum_{n=1}^{\infty} \frac{a^{-1}(n)}{n^k} + \sum_{n=1}^{\infty} \frac{a^{-1}(n)}{n^k} \left(\exp\left({-\frac{ x}{n^{1/A}}}\right) -1 \right) \nonumber \\
&= \frac{1}{\phi(k) }+\sum_{n=1}^{\infty} \frac{a^{-1}(n)}{n^k}\frac{1}{2\pi i} \int_{d-i\infty}^{d+i \infty}\Gamma(s) \left(\frac{ x}{n^{1/A}}\right)^{-s} {\rm d}s  \nonumber \\
& =  \frac{1}{\phi(k)} +\frac{1}{2\pi i}\int_{d-i \infty}^{d+i \infty}\frac{\Gamma(s)}{\phi(k-s/A)} x^{-s} {\rm d}s. 
\end{align}
Since $\Re(k-s/A) > k_0$, hence $1/\phi(k-s/A)=  \sum_{n=1}^{\infty}\frac{a^{-1}(n)}{n^{k-s/A}}$ is absolutely and uniformly convergent in any compact subset of the domain. Thus, interchanging the order of summation and integration is justified. We now intend to solve the following vertical integral, 
\begin{align*}
J_{k, A}(x):= \frac{1}{2\pi i}\int_{d-i \infty}^{d+i \infty} \frac{\Gamma(s)}{\phi(k-s/A)} x^{-s} {\rm d}s.
\end{align*}
We shall investigate the poles of the integrand function in the first place.  One can observe that $\Gamma(s)$ has simple poles at non-positive integers. Let $S$ be the set of all the trivial zeros of $\phi(s)$ and to collect the contribution of the residual term due to non-trivial zeros of $\phi(k-s/A)$, we change the line of integration from $ \Re(s) = d \in (-1,0)$ to $ \Re(s) = \lambda \in (Ak + B, Ak+ B + \epsilon)$ with $0 < \epsilon < 1$. 
For that, we consider a rectangular contour $\mathcal{C}$ determined by the line segments $[ \lambda-iT,  \lambda+iT], [\lambda+i T, d + i T], [d + i T,  d - i T]$, and $ [d - i T, \lambda-iT]$,  where  $T$ is some large positive number. 

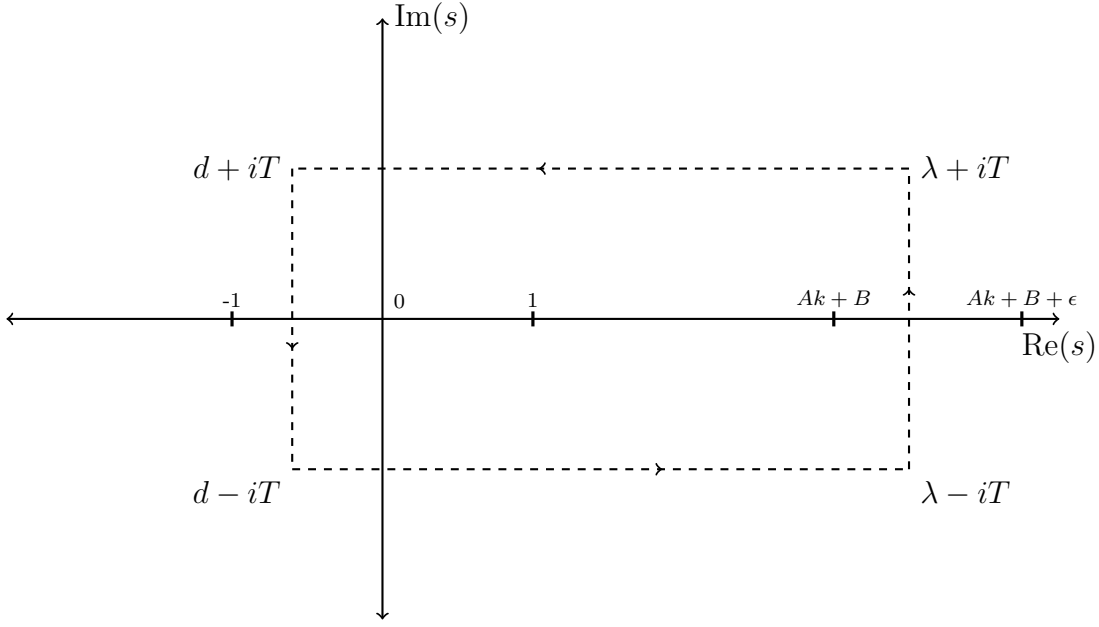
\begin{center}
	\begin{tikzpicture}[very thick,decoration={
			markings,
			mark=at position 0.6 with {\arrow{>}}}] 
		\draw[thick,dashed,postaction={decorate}] (-1.2,-2)--(7,-2) node[below right, black] {$\lambda-i T$};
		\draw[thick,dashed,postaction={decorate}] (7,-2)--(7,2)  node[right, black] {$\lambda+iT$} ;
		\draw[thick,dashed,postaction={decorate}] (7,2)--(-1.2,2) node[left, black] {$d+i T$}; 
		\draw[thick,dashed,postaction={decorate}] (-1.2,2)--(-1.2,-2)  node[below left, black] {$d-i T$}; 
		\draw[thick, <->] (-5,0) -- (9,0) coordinate (xaxis);
		\draw[thick, <->] (-0,-4) -- (-0,4)node[midway, above right, black] {\tiny $0$} coordinate (yaxis);
		\draw (2,0.1)--(2,-0.1) node[midway, above, black] {\tiny1} ;
		%\draw (-1,0.1)--(-1,-0.1) node[midway, above, black] {\tiny-1} ;
		\draw (-2,0.1)--(-2,-0.1) node[midway, above, black] {\tiny-1} ;
		\draw (6,0.1)--(6,-0.1) node[midway, above, black] {\tiny $ Ak+ B$} ;
		\draw (8.5,0.1)--(8.5,-0.1) node[midway, above, black] {\tiny $  Ak+B+\epsilon $} ;
		\node[below] at (xaxis) { $\Re(s)$};
		\node[right]  at (yaxis) { $\Im(s)$};
	\end{tikzpicture}
\end{center}

We now use Cauchy's residue theorem to get
\begin{align}\label{Use of CRT}
    \frac{1}{2\pi i}\int_{\mathcal{C}}\frac{  \Gamma(s)  }{\phi(k-s/A) } x^{-s} {\rm d}s = R_{0} + R_{t} +  \sum_{|\Im(\rho)|<T} {R}_{\rho}(x),
\end{align}
where $R_0$ stands for the residual term corresponding to the simple pole of $\Gamma(s)$ at $s=0$, $R_t $ is the residual term due to trivial zeros of $\phi(k-s/A)$ with $0\leq \Re\left(  k- \frac{s}{A} \right) \leq Ak+B$ and ${R}_{\rho}(x)$ is the residual term corresponding to the non-trivial zero  $\rho$ of $\phi(k-s/A)$ .  Next goal is to show that the horizontal integrals 
\begin{align*}
H_1(T,A) = \frac{1}{2 \pi i } \int_{\lambda + i T}^{d+i T} \frac{  \Gamma(s)  }{\phi(k-s/A) } x^{-s} {\rm d}s, \\
H_2(T,A) = \frac{1}{2 \pi i } \int_{d- i T}^{\lambda - i T} \frac{  \Gamma(s)  }{\phi(k-s/A) } x^{-s} {\rm d}s, 
\end{align*}
will vanish as $T \rightarrow \infty $. This follows if we make use of Stirling's formula for $\Gamma(s)$,  that is,  Lemma \ref{Stirling} along with Lemma \ref{L function bound to vanish horizontal integral}.
Thus,  letting $T\rightarrow \infty$ in \eqref{Use of CRT},  we arrive at 
\begin{align}\label{two vertical integral_for_general_CN}
\frac{1}{2 \pi i } \int_{d - i \infty}^{d+i \infty} \frac{  \Gamma(s)  }{\phi(k-s/A) } x^{-s} {\rm d}s = \frac{1}{2 \pi i } \int_{\lambda - i \infty}^{\lambda+i \infty} \frac{  \Gamma(s)  }{\phi(k-s/A) } x^{-s} {\rm d}s - R_0 - R_t -  \sum_{\rho} {R}_{\rho}(x),
\end{align}
where the sum over $\rho$ runs through all the non-trivial zeros of $\phi(k-s/A)$. The residue at $s=0$ can be evaluated as 
\begin{align}\label{Residue_s=0_for_general_CN}
R_{0} = \lim_{s \rightarrow 0}  \frac{  s \Gamma(s)  }{\phi(k-s/A) } x^{-s}= \frac{1}{\phi(k)}. 
\end{align}
Considering the non-trivial zeros of $\phi(s)$ to be simple, we get
\begin{align}\label{Residue non trivial for general CN}
\sum_{\rho} {R}_{\rho}(x) = \sum_{\rho} \lim_{ s \rightarrow A(k-\rho)} \frac{ \left( s-  A(k-\rho) \right) \Gamma(s)   }{\phi(k-s/A)} x^{-s} = -A \sum_{\rho} \frac{ \Gamma(A(k-\rho))   }{\phi'(\rho)} x^{-A(k-\rho)}.
\end{align}
Substituting \eqref{two vertical integral_for_general_CN}, \eqref{Residue_s=0_for_general_CN} and \eqref{Residue non trivial for general CN} in \eqref{main equation for general CN}, we get  
\begin{align}\label{identity_half_obtained}
\sum_{n=1}^{\infty} \frac{a^{-1}(n)}{n^k} \exp \left({-\frac{ x}{n^{1/A}}}\right) & = \frac{1}{2 \pi i } \int_{\lambda - i \infty}^{\lambda+i \infty} \frac{  \Gamma(s)  }{\phi(k-s/A) } x^{-s} {\rm d}s - R_t \nonumber \\
& +A \sum_{\rho} \frac{ \Gamma(A(k-\rho))}{\phi'(\rho)} x^{-A(k-\rho)}.
\end{align}
We shall now focus on the derivation of the line integral present above.  Let us define
\begin{align}\label{right_vertical_integral_for_general_CN}
I_{k,A}(x) :=  \frac{1}{2 \pi i } \int_{\lambda - i \infty}^{\lambda+i \infty} \frac{  \Gamma(s)  }{\phi\left(k- \frac{s}{A}\right) } x^{-s} {\rm d}s. 
\end{align}
Here, we bring into play the role of the functional equation for $\phi(s)$. Replacing $s$ by $k-s/A$ in \eqref{modified functional equation for phi and psi},  it can be seen that
\begin{align}\label{1/phi for general CN}
\frac{1}{\phi(k- \frac{s}{A})} = \frac{c^{2k-\delta-\frac{2s}{A}} \Gamma(Ak+B-s)}{\nu \Gamma(A\delta - Ak + B + s) \psi(\delta - k +\frac{s}{A})}. 
\end{align}
Utilizing \eqref{1/phi for general CN} in \eqref{right_vertical_integral_for_general_CN}, we get 
\begin{align}\label{another form of right vertical integral for CN}
I_{k,A}(x) = \frac{1}{\nu}  \frac{c^{2k- \delta}}{ 2 \pi i } \int_{\lambda - i \infty}^{\lambda+i \infty} \frac{  \Gamma(s) \Gamma(Ak+B-s) }{\Gamma(A(\delta-k)+B+s)\psi(\delta - k+ \frac{s}{A}) } (c^{2/A}x)^{-s} {\rm d}s. 
\end{align}
Note that $\Re(\delta - k +\frac{s}{A}) > \delta + \frac{B}{A}$. Thus, to use the series expansion of $\frac{1}{\psi(\delta - k+ \frac{s}{A})} = \sum_{n= 1}^\infty \frac{b^{-1}(n)}{n^{s/A}} n^{k-\delta}$ in \eqref{another form of right vertical integral for CN},  constraints us to choose $k_0$ such that $k_0=\delta + \frac{B}{A} $. Then interchanging the order of summation and integration leads us to obtain that,
\begin{align}\label{I_{k,A}}
I_{k,A}(x) = \frac{ c^{2k-\delta}}{\nu} \sum_{n = 1}^\infty \frac{b^{-1}(n)}{ n^{\delta - k}} V_{k,A}(X_n),
\end{align}
where
\begin{align}\label{final form of right vertical integral for general cn}
 V_{k,A}(X_n) = \frac{1}{2\pi i}\int_{\lambda - i \infty}^{\lambda+i \infty} \frac{  \Gamma(s) \Gamma(Ak+B-s) }{\Gamma(A(\delta - k)+ B +s) } X_n^{-s} {\rm d}s,
\end{align}
and $X_n = (c^2 n)^{1/A}x$. We shall now focus on $V_{k,A}(X_n)$ and we aim at writing this vertical integral in terms of the Meijer $G$-function.  However,  it can be seen that the line of integration $ Ak+ B < \Re(s)=\lambda <  Ak+ B + \epsilon$ does not separate the poles of $\Gamma(s)$  from that of $\Gamma(Ak+B-s)$.
Thus,  to use the definition of the Meijer $G$-function \eqref{Meijer-G}, 
we relocate the line of integration from $\Re(s) = \lambda$ to $0< \Re(s) = \beta < Ak + B$ to ensure that the line of integration separates the poles of $\Gamma(s)$  and $\Gamma(Ak+B-s)$.  Finally,  we consider a contour  $\mathcal{D}_1$ determined by the line segments $[\lambda-iT,\lambda+iT],  [\lambda+i T,  \beta+iT], [\beta+iT, \beta-iT]$,  and $[\beta- iT,  \lambda-iT]$.

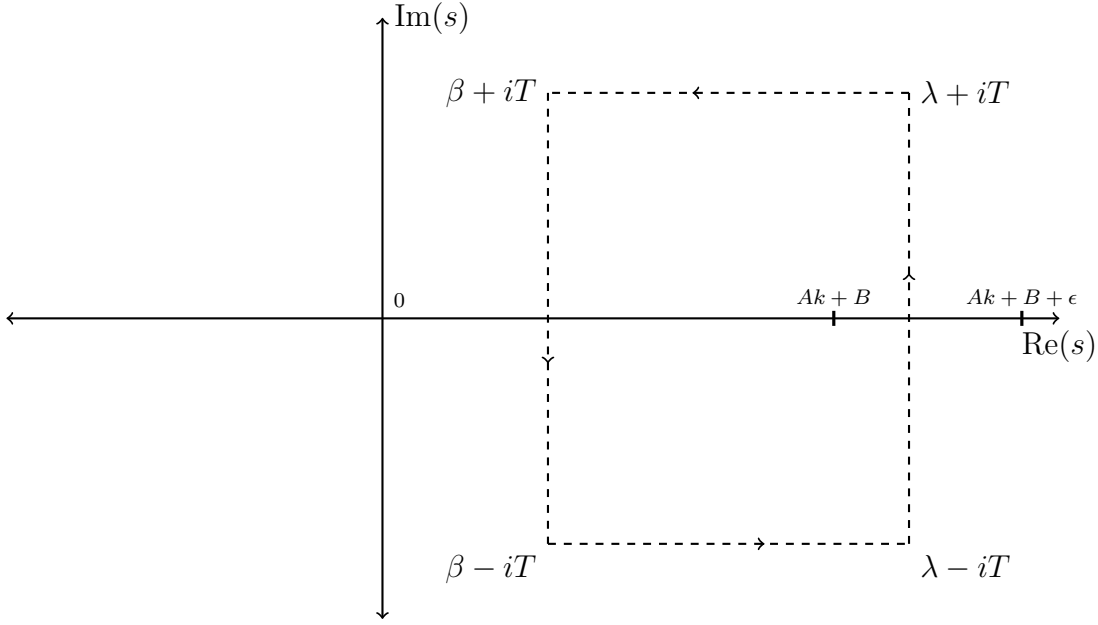
\begin{center}
	\begin{tikzpicture}[very thick,decoration={
			markings,
			mark=at position 0.6 with {\arrow{>}}}] 
		\draw[thick,dashed,postaction={decorate}] (2.2,-3)--(7,-3) node[below right, black] {$\lambda-i T$};
		\draw[thick,dashed,postaction={decorate}] (7,-3)--(7,3)  node[right, black] {$\lambda+iT$} ;
		\draw[thick,dashed,postaction={decorate}] (7,3)--(2.2,3) node[left, black] {$\beta+i T$}; 
		\draw[thick,dashed,postaction={decorate}] (2.2,3)--(2.2,-3)  node[below left, black] {$\beta-i T$}; 
		\draw[thick, <->] (-5,0) -- (9,0) coordinate (xaxis);
		\draw[thick, <->] (-0,-4) -- (-0,4)node[midway, above right, black] {\tiny $0$} coordinate (yaxis);
		%\draw (2,0.1)--(2,-0.1) node[midway, above, black] {\tiny1} ;
		%\draw (-1,0.1)--(-1,-0.1) node[midway, above, black] {\tiny-1} ;
		%\draw (-2,0.1)--(-2,-0.1) node[midway, above, black] {\tiny-1} ;
		\draw (6,0.1)--(6,-0.1) node[midway, above, black] {\tiny $Ak+ B$} ;
		\draw (8.5,0.1)--(8.5,-0.1) node[midway, above, black] {\tiny $Ak+ B+\epsilon $} ;
		\node[below] at (xaxis) {$\Re(s)$};
		\node[right]  at (yaxis) {$\Im(s)$};
	\end{tikzpicture}
%	{\bf Fig:  } 
\end{center} 

Again,  utilizing the Cauchy's residue theorem,  we have
\begin{align}\label{2nd_application_CRT_for_general_CN}
\frac{1}{2 \pi i} \left(  \int_{\lambda-iT}^{\lambda+iT} + \int_{\lambda+i T}^{\beta + i T} + \int_{\beta + i T}^{\beta - i T} + \int_{\beta - i T}^{\lambda - i T}\right)  \frac{  \Gamma(s) \Gamma(Ak+B-s) }{\Gamma(A(\delta - k)+ B +s) } X_n^{-s} {\rm d}s = R_{Ak+B},
\end{align}
where $R_{Ak+B}$ denotes the residual term due to the simple pole of $\Gamma(Ak+B-s)$ at $s = Ak+B$.  We can easily verify that 
\begin{align}\label{Residue at Ak+B}
R_{Ak+B} = -\frac{\Gamma(Ak+B)}{\Gamma(A\delta + 2B)}X_n^{-Ak-B}.
\end{align}
Letting  $T \rightarrow \infty$,  one can show that horizontal integrals vanish.  Thus substituting \eqref{Residue at Ak+B} in \eqref{2nd_application_CRT_for_general_CN} and using \eqref{final form of right vertical integral for general cn}, it  leads us to obtain that
\begin{align}\label{After CRT}
 V_{k,A}(X_n) =   \frac{1}{2\pi i} \int_{\beta - i \infty}^{\beta +i \infty} \frac{  \Gamma(s) \Gamma(Ak+B-s) }{\Gamma(A(\delta - k)+ B +s) } X_n^{-s} {\rm d}s-\frac{\Gamma(Ak+B)}{\Gamma(A\delta + 2B)}X_n^{-Ak-B}.
\end{align}
As we know that the line of integration $\Re(s) = \beta$ separates the poles of $\Gamma(s)$ from the poles of $\Gamma(Ak+B-s)$,  so we are all set to write the vertical integral,  in \eqref{After CRT},  in terms of the Meijer $G$-function \eqref{Meijer-G} with $n=p=m=1, q=2$ and $a_{1} = 1, b_{1} = Ak+B$ and $ b_{2} = 1-A\delta+Ak-B$. Thus, we have 
\begin{align}\label{in terms of Meijer G}
 \frac{1}{2 \pi i}\int_{\beta - i \infty}^{\beta + i \infty}   \frac{  \Gamma(s) \Gamma(Ak+B-s) }{\Gamma(A(\delta - k)+ B +s) } X_n^{-s} {\rm d}s   = G_{1,2}^{1,1} \left(\begin{matrix} 1 \\ Ak+B,  1-A\delta+Ak-B  \end{matrix} \Big| \frac{1}{X_n}\right).
\end{align}
Here one can easily check that $p+q < 2(m+n)$ and  $|\arg(\frac{1}{X_n})| < \left(m+n - \frac{p+q}{2} \right) \pi$. Thus all the necessary conditions for the convergence of the above Meijer $G$-function are satisfied.  Now invoking Slater's theorem \eqref{Slater}, one can see that
\begin{align}\label{use of slater for general cn}
 G_{1,2}^{1,1} \left(\begin{matrix} 1 \\ Ak+B,  1-A\delta+Ak-B  \end{matrix} \Big| \frac{1}{X_n}\right) = \frac{1}{ X_n^{Ak+B}} \frac{\Gamma(Ak+B)}{\Gamma(A\delta+ 2 B )}  {}_1F_{1} \left( \begin{matrix}
Ak+B \\ A\delta+2B  \end{matrix}  \Big| -\frac{1}{X_n} \right).
\end{align}
Finally,  substituting \eqref{use of slater for general cn} in \eqref{in terms of Meijer G} and together with \eqref{After CRT},  \eqref{I_{k,A}},   \eqref{right_vertical_integral_for_general_CN} and then combining all these terms in \eqref{identity_half_obtained},  we  complete the proof of Theorem \ref{General_Identity}.
\end{proof}

\begin{proof}[Corollary \ref{general_alpha_beta_identity}][]
Letting $k = \delta + \frac{B}{A}$ and $x=\beta >0$ with $\alpha \beta = \frac{1}{c^{2/A}}$ in Theorem \ref{General_Identity} and then multiplying by $\sqrt{\nu}\beta^{\frac{A\delta +2 B}{B}}$ throughout and using the fact that $|\nu|=1$,    one can complete the proof of this result.  
\end{proof}

\begin{proof}[Corollary \ref{corolary for r_2(n)}][]
The identity \eqref{main identity for r_2(n)} represents a particular instance of the identity \eqref{main identity} with $a(n) = b(n) = r_2(n)$,  where $r_2(n)$ is defined in Example \ref{r_b(n)} for $j = 2$.  Note that $\zeta_{id}(2,s)$ satisfies the following functional equation 
\begin{align*}
\pi^{-s} \Gamma(s) \zeta_{id}(2,s) = \pi^{-(1-s)} \Gamma\left(1- s\right) \zeta_{id}\left(2,1- s\right). 
\end{align*}
This corresponds to \eqref{modified functional equation for phi and psi} for $a(n) = b(n) = r_{2}(n),  c = \frac{1}{\pi}, A=1, B=0, \delta = 1$ and $\nu = 1.$ Now substituting these values in Theorem \ref{General_Identity} and noting that the term $R_t$ is zero as we do not encounter any trivial zeros of $\zeta_{id}(2,k-s)$ with $0 \leq \Re(k-s)\leq k$,  and also $\sum_{n=1}^{\infty} \frac{r_2^{-1}(n)}{n}=0$ since we know the relation \eqref{Epstein in terms of zeta},  we complete the proof.
\end{proof}

\begin{proof}[Corollary \ref{corollary for dedeking zeta}][]
This scenario represents a special case of \eqref{main identity} for the Dedekind zeta function for the imaginary quadratic field $Q(\sqrt{-D})$,  which is defined in Example \ref{example dedekind}.  From the functional equation \eqref{fn eq dedekind},  it is clear that $c= \frac{\sqrt{D}}{2\pi}, A=1, B=0, \delta = 1$ and $\nu = 1$.  Now putting these values in Theorem \ref{General_Identity} and  observing that $R_t=0$ since we do not encounter a trivial zero of $\zeta_{Q(\sqrt{-D})}(k-s)$ with  $0 \leq \Re(k-s)\leq k$,   and $\sum_{n=1}^{\infty} \frac{\mathsf{a}_n^{-1}}{n}=0$,  one can complete the proof. 
\end{proof}

\begin{proof}[Theorem \ref{corollary for zeta(s) zeta(s-r)}][]
This identity is a special case of the identity \eqref{main identity} with  $a(n) = b(n) = \sigma_r(n),  c = \frac{1}{2\pi}, A=1, B=0, \delta = r+1$ and $\nu = (-1)^{\frac{r+1}{2}}$ follows from Example \ref{def zeta(s) zeta(s-r)}.  
In this case $\phi(s)= \zeta(s) \zeta(s-r)$,  where $r \geq 1$ odd integer.  Thus,  the trivial zeros of $\phi(s)$ are at $\{ -2m,  r-2m \}$,  with $m \in \mathbb{N}$.  Hence, the trivial zeros of $\phi(k-s)$ with $0 \leq \Re(k-s) \leq k$ are at $k-r+2m$ such that  $ 1 \leq m \leq \lfloor \frac{r}{2} \rfloor$.  
Therefore,  from \eqref{Use of CRT},  one can see that the residual term $R_t$ becomes
{\allowdisplaybreaks \begin{align}
R_t & =  \sum_{m=1}^{ \lfloor \frac{r}{2} \rfloor} \lim_{s \rightarrow k-r+2m} \frac{(s- (k-r+2m))  \Gamma(s) x^{-s}}{\zeta(k-s)\zeta(k-s-r)} \nonumber  \\
& = - \sum_{m=1}^{ \lfloor \frac{r}{2} \rfloor}  \frac{\Gamma(k-r+2m)x^{r-k-2m}}{\zeta(r-2m)\zeta^{'}(-2m)} \nonumber \\
&  = \sum_{m=1}^{[\frac{r}{2}]} \frac{\Gamma(k-r+2m)}{\zeta(r-2m) \zeta(2m+1) } \frac{(-1)^m 2(2\pi)^{2m}}{(2m)!} x^{r-k-2m}.  \label{Final R_t}
\end{align} }
In the final step we utilized the identity
\begin{align*}
\zeta^{'}(-2m) = \frac{(-1)^m (2m)! \zeta(2m+1)}{2(2\pi)^{2m}}.
\end{align*}
From \eqref{Use of CRT}, one can see that the integrand function $\frac{\Gamma(s)x^{-s}}{\phi(k-s/A)}$,  and in this case $A=1$ and  $\phi(k-s)= \zeta(k-s) \zeta(k-s-r)$,  so the residual term $R_{\rho}(x)$ will depend on the non-trivial zeros of $\zeta(k-s)$ as well as $\zeta(k-s-r)$.  Thus,  if we assume that the non-trivial zeros of $\zeta(s)$ are simple,  then one can show that  
\begin{align}\label{R_rho}
\sum_{\rho} R_{\rho}(x) = -\sum_{\rho} \frac{x^{\rho-k}}{\zeta^{'}(\rho)} \bigg[\frac{\Gamma(k-\rho)}{\zeta(\rho -r)} + \frac{x^r \Gamma(k-r-\rho)}{\zeta(r+\rho)}\bigg],
\end{align}
where the sum over $\rho$ runs through the non-trivial zeros of $\zeta(s)$.  Again note that $\phi(s)=\zeta(s)\zeta(s-r)$ has a simple pole at $1+r$,  thus  we will have $\sum_{n=1}^\infty  \frac{\sigma^{-1}_{r}(n)}{n^{r+1}}=0$.  Substituting the above values \eqref{Final R_t}-\eqref{R_rho} of $R_t$ and $R_{\rho}(x)$ in Theorem \ref{General_Identity},  we complete the proof of Theorem \ref{corollary for zeta(s) zeta(s-r)}.   
\end{proof}

\begin{proof}[Corollary \ref{alpha-beta form for sigma}][]
Substituting $k=r+1$,  $x=\alpha$, $\alpha \beta = 4\pi^2$ and multiplying by $\alpha^\frac{r+1}{2}$ on both sides of Theorem \ref{corollary for zeta(s) zeta(s-r)},  we obtain the result.  
\end{proof}

\begin{proof}[Theorem \ref{corollary for cusp form}][]
This identity represents a particular instance of \eqref{General_Identity}, for the normalized Fourier coefficients of the holomorphic Hecke eigenform $f(z)$. Here, we take $a(n) = b(n) =  \lambda_f(n) $ and the corresponding $L$-function $L(f,s)$ is defined in Example \ref{Def cusp}. In this case, we encounter a pole of the integrand $\frac{\Gamma(s)x^{-s}}{L(f,k-s)}$ due to the trivial zero of $L(f,  k-s)$ at $s = k+ \frac{\omega -1}{2}$ which satisfies the condition $0 \leq \Re(k-s)\leq k+ \frac{w-1}{2}$. Thus, we get our residual term $R_t$ to be 
\begin{align*}
R_t &= \lim_{s \rightarrow k+\frac{\omega -1}{2}} \frac{\left(s-\left( k +\frac{\omega -1}{2} \right) \right)\Gamma(s)x^{-s}}{L(f, k-s)} \\
&= \lim_{s \rightarrow k+\frac{\omega -1}{2}} \frac{\left(s-\left( k +\frac{\omega -1}{2} \right) \right)\Gamma\left(k-s+\frac{\omega -1}{2}\right)\Gamma(s)(2\pi)^{1+2s-2k}x^{-s}}{\Gamma\left(1-k+s+\frac{\omega -1}{2}\right)L(f,1-k+s)} \\
& = -\frac{(2\pi)^{\omega}}{x^{k+\frac{\omega -1}{2}}} \frac{\Gamma\left(k+\frac{\omega -1}{2}\right)}{\Gamma(\omega)} \frac{1}{L\left(f,\frac{\omega +1}{2}\right)}.
\end{align*}
Here, in the penultimate step, we have used functional equation \eqref{fn eq cusp form} for $L(f,s)$. Substituting this value of $R_t$ along with $ c = \frac{1}{2\pi}, A=1, B=\frac{\omega -1}{2}$, $ \delta = 1$ and $\nu = 1$ in Theorem \ref{main identity} and using $\frac{1}{L\left(f,\frac{\omega +1}{2}\right)} = \sum_{n=1}^{\infty}\frac{\lambda_f^{-1}(n)}{n^{\frac{\omega +1}{2}}}$, we obtain 
\eqref{main identity for cusp form}.
\end{proof}

\begin{proof}[Corollary \ref{alpha beta form for cusp forms}][]
Letting $k = \frac{\omega +1}{2}$, $x = \beta$,  $\alpha \beta = 4\pi^2 $ in Theorem \ref{corollary for cusp form} and multiplying throughout by $\beta^{\frac{\omega}{2}}$, we obtain the result.  
\end{proof}

\begin{proof}[Corollary \ref{corollary for Ramanujan tau function}][]
This result is an immediate implication of Corollary \ref{alpha beta form for cusp forms} for the Ramanujan cusp form of weight $12$.  Thus,  substituting $\lambda_f(n)=\tau_0(n)$ and $\omega=12$,  we get our modular relation \eqref{identity for Ramanujan tau}.
\end{proof}

\begin{proof}[Theorem \ref{conjecture is true}][]

Let $L(f,s)$ be a nice $L$-function defined as in  \eqref{def nyc L function}. 
First,  we begin by presupposing that the grand Riemann hypothesis (GRH) for $L(f,s)$ holds true. Then from \eqref{merten bound general},  for any $\epsilon >0$,  we have
\begin{align}\label{merten bound general1}
G(x) := \sum_{ 1 \leq n \leq x} \mu_f(n) \ll_{\epsilon, f} x^{\frac{1}{2} + \epsilon}.
\end{align}
Now, we employ Euler's partial summation formula,  i.e.,  Lemma \ref{Euler's summation} with $a(n) =  \mu_f(n)$ and $f(t) = t^{-k}$ to see that
\begin{align}\label{H(m,n)}
H(m,n) := \sum_{i=m}^{n} \frac{\mu_f(i)}{i^k} & =  G(n) f(n)- G(m-1) f(m-1) - \int_{m-1}^n G(t) f'(t) {\rm d}t.
\end{align}
Utilizing \eqref{merten bound general1} in \eqref{H(m,n)}, one can obtain 
\begin{align}\label{bound H(m,n)}
H(m,n) = O_{\epsilon,k}\left(m^{\frac{1}{2}-k+\epsilon}\right).
\end{align}
The above bound is uniform in $n$. Our main objective is to determine the bound for the following infinite series,  under the assumption of the GRH for $L(f,s)$,  
\begin{align*}
P_{k,f,\ell}(x) :=   \sum_{n=1}^{\infty} \frac{\mu_f(n)}{n^{k}} \exp \left(- \frac{ x}{n^{\ell}}\right) .
\end{align*}
For simplicity, we replace $x$ by $x^{\ell}$ and separate the sum into finite and infinite part as 
\begin{align}\label{Q_1 + Q_2}
P_{k,f,\ell}(x^{\ell}) &= \sum_{n=1}^{m -1}\frac{\mu_f(n)}{n^{k}} \exp \left(- \frac{ x^{\ell}}{n^{\ell}}\right) + \sum_{n=m}^{\infty}\frac{\mu_f(n)}{n^{k}} \exp \left(- \frac{ x^{\ell}}{n^{\ell}}\right) \nonumber \\
&:= Q_1(x^{\ell}) + Q_2(x^{\ell}).
\end{align}
where $ m = [x^{1-\epsilon}]+1$.   To get an estimate for $Q_1(x^{\ell})$, we write 
\begin{align*}
|Q_1(x^{\ell})| = \bigg|\sum_{n=1}^{m -1}\frac{\mu_f(n)}{n^{k}} \exp \left(- \frac{ x^{\ell}}{n^{\ell}}\right) \bigg| & \leq \sum_{n=1}^{m -1} |H(n,n)| \exp\left(-\frac{x^{\ell}}{m^{\ell}}\right) \\
& \ll \sum_{n=1}^{m-1} n^{\frac{1}{2}-k+\epsilon} \exp\left(-\frac{x^{\ell}}{m^{\ell}}\right) \\
& \ll_{k, \epsilon} m^{\frac{3}{2}-k+\epsilon} \exp\left(-\frac{x^{\ell}}{m^{\ell}}\right). 
\end{align*}
This gives 
\begin{align}\label{bound Q_1}
Q_1(x^{\ell}) = O\left(\exp (-x^{\ell \epsilon}) x^{(1-\epsilon)(\frac{3}{2}-k+\epsilon)} \right),
\end{align}
as $m - 1 = [x^{1-\epsilon}]$.
We shall now try to estimate $Q_2(x^{\ell})$. 
For sufficiently large integer $N > m$, we write 
\begin{align*}
\sum_{n=m}^{N} \frac{\mu_f(n)}{n^{k}} \exp \left(- \frac{ x^{\ell}}{n^{\ell}}\right)   &= \sum_{n=m}^{N-1} H(m,n) \left(  \exp\left(-\frac{x^{\ell}}{n^{\ell}}\right) -  \exp\left(-\frac{x^{\ell}}{(n+1)^{\ell}}\right)\right) \nonumber \\ & + H(m,n) \exp\left(-\frac{x^{\ell}}{N^{\ell}}\right).
\end{align*}
Allowing $N \rightarrow \infty $ and invoking the bound \eqref{bound H(m,n)} for $H(m,n)$, it can be inferred that 
\begin{align}\label{Q_2}
Q_2(x^{\ell})= Q_3(x^{\ell}) +  O_{\epsilon,k}\left(m^{\frac{1}{2}-k+\epsilon}\right),
\end{align}
where 
\begin{align*}
Q_3(x^{\ell}) =  \sum_{n=m}^{N-1} H(m,n) \left(  \exp\left(-\frac{x^{\ell}}{n^{\ell}}\right) -  \exp\left(-\frac{x^{\ell}}{(n+1)^{\ell}}\right)\right).
\end{align*}
Now to simplify further, we use Cauchy's mean value theorem with the function $T(z) =  \exp\left(-\frac{x^{\ell}}{z^{\ell}}\right).$ One can find $z_n \in (n,n+1)$ such that $$T(n) - T(n+1) = -T^{'}(z_n) = - \frac{\ell x^{\ell}}{z_n^{\ell +1}}  \exp\left(-\frac{x^{\ell}}{z_n^{\ell}}\right).$$ 
Employing the above fact in conjunction with \eqref{bound H(m,n)}, we obtain
{\allowdisplaybreaks \begin{align}\label{bound Q_3}
|Q_3(x^{\ell})| 
&\ll_{\epsilon,\ell} m^{\frac{1}{2}-k+\epsilon}  \sum_{n=m}^{\infty} \frac{x^\ell}{z_n^{\ell+1}}  \exp\left(-\frac{x^{\ell}}{z_n^{\ell}}\right) \nonumber \\ 
& \ll_{\epsilon, \ell}  m^{\frac{1}{2}-k+\epsilon} \sum_{n=m}^\infty \frac{x^{\ell}}{n^{\ell+1}} \nonumber \\ 
& \ll_{\epsilon, \ell} x^{\frac{1}{2}-k+\epsilon} \frac{x^\ell}{m^\ell}  \ll_{\epsilon, \ell} x^{\frac{1}{2}-k+\epsilon'}
\end{align}}
as $m \sim x^{1-\epsilon}$.  Now utilizing \eqref{bound Q_3} in \eqref{Q_2}, we see that 
\begin{align}\label{bound Q_2}
|Q_2(x^{\ell})| = O_{\epsilon,k, \ell}\left(m^{\frac{1}{2}-k+\epsilon}\right).  
\end{align}
Finally, substituting the bounds \eqref{bound Q_1} and \eqref{bound Q_2} in \eqref{Q_1 + Q_2}, one can see that due to exponential decay, the bound for $Q_1(x^{\ell})$ goes to zero much faster than the bound for $Q_2(x^{\ell})$ as $x \rightarrow \infty $, and therefore
\begin{align*}
P_{k,f, \ell}(x^{\ell}) = O_{\epsilon,k, \ell}\left(x^{\frac{1}{2}-k+\epsilon}\right).
\end{align*}
Replace $x$ by $x^{1/\ell}$,  we get the desired bound \eqref{conjecture main bound}.

Now, turning to the converse aspect, that is, assuming the validity of the bound \eqref{conjecture main bound}, we aim to show that the GRH for $L(f,s)$ follows. 
Utilizing Lemma \ref{main lemma for converse part}, we deduced that 
\begin{align}\label{idenity for extending region in general}
 L(f, \ell s+k)\int_{0}^{\infty}x^{-s-1}P_{k,f,\ell}(x) \mathrm{d}x = \Gamma(-s),
\end{align}
valid for $\Re(s) > \frac{1-k}{\ell}$ except for non-negative integers.   We aim to  extend the validity of \eqref{idenity for extending region in general} towards the left half plane. To accomplish this, we choose a sufficiently large positive real number $R$ and rewrite the expression \eqref{idenity for extending region in general} as 
\begin{align*}
 L(f, \ell s+k)\left(\int_{0}^{R}+ \int_{R}^{\infty} \right)x^{-s-1}P_{k,f,\ell}(x) \mathrm{d}x = \Gamma(-s).
\end{align*}
Utilizing the bound \eqref{conjecture main bound} for $P_{k,f,\ell}(x)$, it is evident that unbounded part possesses analyticity for $\Re(s) > \frac{1}{2\ell} - \frac{k}{\ell}$. 
Consequently, identity \eqref{idenity for extending region in general} is analytic in the strip $\frac{1}{2\ell} - \frac{k}{\ell} < \Re(s) < \frac{1-k}{\ell} $. 
 Considering the fact that $\Gamma(s)$ never attains zero, it implies that $L(f,\ell s+k)$ is non-vanishing in the strip $\frac{1}{2\ell} - \frac{k}{\ell} < \Re(s) < \frac{1-k}{\ell} $,  which in turn implies that $L(f,s)$ does not possess any zero in the strip $\frac{1}{2} < \Re(s) < 1$. 
From the symmetry of the functional equation, one can conclude that $L(f,s)$ is non-vanishing in $0 < \Re(s) < \frac{1}{2}$.  This proves that all the non-trivial zeros of $L(f,s)$ lies on the critical line $\Re(s) = \frac{1}{2}$. 
Thus, the GRH follows.  This completes the proof of Conjecture \ref{conjecture}.  
%{\bf Do we need assume that $L(f,s)$ does not vanish on $\Re(s)=1$.  }
\end{proof}

%\begin{proof}[Corollary \ref{equivalent criteria for RH  for epstein}][]
%It is easily follows from the proof of Theorem \ref{conjecture is true}. Thus, we omit the details.
%\end{proof}
%
%************************************
%
%
%
%
%
%
%
%\begin{proof}[Corollary \ref{equivalent criteria for RH  for hecke}][]
%Deriving directly from the proof of Theorem \ref{conjecture is true}, we opt to skip the intricacies.
%\end{proof}

\begin{proof}[Theorem \ref{Equivalent criteria for RH by sigma}][]
For the direct part, we assume RH for $\zeta(s)$ to be true and   we wish to obtain the bound \eqref{bound for eq criteria for sigma} for $P_{k,\sigma,\ell}(x)$. The proof goes along the similar line of reasoning as in the proof of Theorem \ref{conjecture is true} except for that fact that we start with 
\begin{align*}
 \sum_{n \leq x} \sigma_{r}^{-1}(n) = O(x^{\frac{1}{2}+ r + \epsilon}),
\end{align*}
as established in Lemma \ref{mertens bound for sigma} under the assumption of RH. Thus, we omit the proof here for the direct part.\\
For the converse part, assuming the validity of the bound \eqref{bound for eq criteria for sigma} for $P_{k,\sigma,\ell}(x)$,  that is, 
 \begin{align*}
P_{k,\sigma,\ell}(x) := \sum_{n=1}^{\infty}\frac{\sigma_r^{-1}(n)}{n^k} \exp\left(-\frac{x}{n^{\ell}}\right) = O_{\epsilon, k,r, \ell}\left(x^{-\frac{k}{\ell} +\frac{1+2r}{2 \ell}+\epsilon}\right),
\end{align*}
for any $\epsilon > 0$, 
we will try to show that all the non-trivial zeros of $\zeta(s)$ are located on the critical line $\Re(s) = \frac{1}{2}$. Here, we utilize Lemma \ref{main lemma for sigma bound} to deduce that
\begin{align}\label{identity for extending region for sigma}
\zeta(k+\ell s)\zeta(k+\ell s-r)\int_{0}^{\infty}x^{-s-1}P_{k,\sigma,\ell}(x) \mathrm{d}x = \Gamma(-s), 
\end{align}
holds within the region $\Re(s) > \frac{1-k+r}{\ell}$ except for $s = 0,1,2 \cdots$. Continuing along the path drawn in the proof of Theorem \ref{conjecture is true},  one can extend the region of analyticity of \eqref{identity for extending region for sigma} in  $\frac{1}{2 \ell} - \frac{k}{\ell} + \frac{r}{\ell} < \Re(s) < \frac{1-k+r}{\ell}.$  Considering the fact that $\Gamma(s)$ never attains zero, it implies that $\zeta(k+\ell s)$ and $\zeta(k+\ell s-r)$ is non-vanishing in the strip $\frac{1}{2 \ell} - \frac{k}{\ell} + \frac{r}{\ell} < \Re(s) < \frac{1-k+r}{\ell}.$ Since $r$ is a non-negative integer,  so $\Re(k+ \ell s) > 1$ and $\zeta(k+\ell s)$ is non-vanishing in the associated region and as a result, we are unable to derive any useful information. However,  the factor $\zeta(k+\ell s-r)$ should also be non-vanishing in the region $\frac{1}{2} < \Re(k+\ell s-r) <1$,  which in turn implies that $\zeta(s)$ does not possess any zeros within the strip $\frac{1}{2} < \Re(s) < 1$ and thus the functional equation for $\zeta(s)$ implies that it is non-vanishing in the strip $0< \Re(s) < \frac{1}{2} $. This proves that all the non-trivial zeros of $\zeta(s)$ must lie on the line $\Re(s) = \frac{1}{2} $.
\end{proof}

\section{Numerical evidences for Corollary \ref{corolary for r_2(n)} and Theorem \ref{corollary for zeta(s) zeta(s-r)} }
We used Mathematica software to make the following tables.  To obtain this data,  we considered the left-hand side sum over $n$ with $200$ terms and right-hand side sum over $n$ with $200000$ terms. 

{\allowdisplaybreaks
\begin{table}[h]

\caption{Verification of Corollary \ref{corolary for r_2(n)} }\label{Table for r_2(n)}
\renewcommand{\arraystretch}{1}
{\begin{tabular}{|l|l|l|l|}
\hline
$k$ & $x$ &  Left-hand side  & Right-hand side  \\ 
\hline
 $	2$&  $e+1$ &  $-0.0577422$ &$-0.0577021$\\      
\hline
$3$&$\pi+1$& $-0.0554663$ & $-0.0554081$  \\
\hline
$	6$	&$\pi^2$ & $-0.0000785321$&$-0.0000787028$ \\
\hline
$7$	& $e^{3}$&$-8.02103 \times 10^{-7}$ &$-8.02821 \times 10^{-7}$ \\
\hline
$11$& $e^2+\pi$&   $ 0.000024177$  & $0.00002415$ \\
\hline
\end{tabular}}
\end{table}

\begin{table}[h]
\caption{Verification of Theorem \ref{corollary for zeta(s) zeta(s-r)} }\label{Table of main theorem}
\renewcommand{\arraystretch}{1}
{\begin{tabular}{|l|l|l|l|l|}
\hline
$k$ & $r$ & $x$ &  Left-hand side  & Right-hand side  \\ 
\hline
 $	8$& $5$& $e$ &  $0.0147028$ &$0.0147079$\\      
\hline
$4$&$1 $&$\pi$& $-0.0103086$ & $-0.0103073$  \\
\hline
$	10$	&$1$ &$\pi +1$ & $0.0155115$&$0.0155119$ \\
\hline
$11$	&$5$ & $e+1$&$0.0213461$ &$0.0213510$ \\
\hline
$15$& $7$& $\pi^2$&   $0.0000174587$  & $0.0000174586$ \\
\hline
\end{tabular}}
\end{table}}
\hspace{1cm}

\section{Acknowledgement}  The first author's research is funded by the Prime Minister Research Fellowship,  Govt. of India,   Grant No.  2101705.  The last author wants to thank Science and Engineering Research Board (SERB), India, for giving MATRICS grant (File No. MTR/2022/000545) and SERB CRG grant (File No. CRG/2023/002122). Both authors sincerely thank IIT Indore for providing conductive research environment.

\end{document}